\documentclass{amsart}
\usepackage{amsfonts,latexsym,rawfonts,amsmath,amssymb, amsthm}
\usepackage{latexsym,lscape,rawfonts}
\usepackage{mathrsfs}
\usepackage{galois}
\usepackage[all,cmtip]{xy}
\usepackage{extarrows,color}

\usepackage{times}

\newtheorem{thm}{Theorem}[section]
\newtheorem*{fact1}{Fact 1}%¶¨Àí²»±àºÅ
\newtheorem*{fact2}{Fact 2}%¶¨Àí²»±àºÅ
\newtheorem*{fact3}{Fact 3}%¶¨Àí²»±àºÅ
\newtheorem*{fixed point criterion}{Fixed point criterion}
\newtheorem{cor}[thm]{Corollary}
\newtheorem{lem}[thm]{Lemma}
\newtheorem{prop}[thm]{Proposition}

\newtheorem{conj}[thm]{Conjecture}
\theoremstyle{definition}
\newtheorem{defn}[thm]{Definition}
\newtheorem{exam}[thm]{Example}

\theoremstyle{remark}
\newtheorem{rem}[thm]{Remark}
\numberwithin{equation}{section}
\newcommand{\Z}{\mathbb Z}

\newcommand{\fix}{\mathrm{Fix}\,}

\newcommand{\ind}{\mathrm{ind}}
\newcommand{\chr}{\mathrm{chr}}
\newcommand{\ichr}{\mathrm{ichr}}
\newcommand{\rk}{\mathrm{rk}}
\newcommand{\E}{\mathrm{E}}
\newcommand{\V}{\mathrm{V}}

\newcommand{\aut}{\mathrm{Aut}}
\newcommand{\edo}{\mathrm{End}}
\newcommand{\inj}{\mathrm{Inj}}
\newcommand{\inn}{\mathrm{Inn}}

\newcommand{\stab}{\mathrm{Stab}}
\newcommand{\fpc}{\mathrm{Fpc}}

\newcommand{\tr}{\mathrm{tr}}

\newcommand{\F}{\mathbf{F}}      %fixed point class F

\newcommand{\B}{\mathcal{B}}

\begin{document}

\title[Fixed point indices and fixed words at infinity]{Fixed point indices and fixed words at infinity\\ of selfmaps of graphs}
\author{Qiang Zhang, Xuezhi Zhao}
\address{School of Mathematics and Statistics, Xi'an Jiaotong University,
Xi'an 710049, China}
\email{zhangq.math@mail.xjtu.edu.cn}
\address{Department of Mathematics, Capital Normal University, Beijing 100048, China}
\email{zhaoxve@mail.cnu.edu.cn}

%\thanks{The first author is partially supported by NSFC (No. 11771345), and the second author is partially supported by NSFC (No. 11431009 and No. 11661131004).}
\thanks{The authors are partially supported by NSFC (Nos. 11771345, 11961131004 and 11971389) and Capacity Building for Sci-Tech Innovation-Fundamental Scientific Research Funds.}

\subjclass[2010]{55M20, 55N10, 32Q45}

\keywords{Index, attracting fixed point, fixed subgroups, graph selfmap, free group}

\date{\today}%
%\dedicatory{Dedicated to Professor Boju Jiang on his 80th birthday}
%\commby{}%
% ----------------------------------------------------------------

\begin{abstract}
Indices of fixed point classes play a central role in Nielsen fixed point theory. Jiang-Wang-Zhang proved that for selfmaps of graphs and surfaces, the index of any fixed point class has an upper bound called its characteristic.

In this paper, we study the difference between the index and the characteristic for selfmaps of graphs. First, for free groups, we extend attracting fixed words at infinity of automorphisms into that of injective endomorphisms. Then, by using relative train track technique, we show that the difference mentioned above is quite likely to be the number of equivalence classes of attracting fixed words of the endomorphism induced on the fundamental group.
Since both of attracting fixed words and the existed characteristic are totally determined by endomorphisms themselves, we give a new algebraic approach to estimate indices of fixed point classes of graph selfmaps.

As consequence, we obtain an upper bound for attracting fixed words of injective endomorphisms of free groups, generalizing the one for automorphisms due to Gaboriau-Jaeger-Levitt-Lustig. Furthermore, we give a simple approach to roughly detecting whether fixed words exist or not.
\end{abstract}

\maketitle

% ----------------------------------------------------------------

\section{Introduction}\label{Sect. introduction}

Fixed point theory studies the fixed points of a selfmap $f$ of a space $X$. Nielsen fixed point theory, in particular, is concerned with the properties of the fixed point set
$$\fix f:=\{x\in X|f(x)=x\},$$
that are invariant under homotopy of the selfmap $f$ (see \cite{fp1} for an introduction to Nielsen fixed point theory).

The fixed point set $\fix f$ splits into a disjoint union of \emph{fixed point classes}: two fixed
points $x$ and $x'$ are said to be in the same class if and only if they can be joined by a \emph{Nielsen path} which is
a path homotopic (relative to endpoints) to its own $f$-image. Let $\fpc(f)$ denote the set of all the fixed point classes of $f$. For each fixed point class $\F\in \fpc(f)$, a homotopy invariant \emph{index} $\ind(f,\F)\in \Z$ is well-defined. A fixed point class is \emph{essential} if its index is non-zero. For alternative definitions of fixed point class by covering spaces, see Section \ref{defn Fpc by covering}.

In the paper \cite{JWZ}, B. Jiang, S. Wang and Q. Zhang defined a new homotopy invariant \emph{rank} $\rk(f,\F)\in \mathbb{N}$ for a fixed point class $\F$ of $f$. For an endomorphism $\phi:G \rightarrow G$ of a group $G$, its \emph{fixed subgroup} refers to the subgroup
$$\fix(\phi):=\{g\in G|\phi(g)=g\}\leq G.$$
The \emph{stabilizer} of a fixed point $x\in \F$ is the subgroup
$\stab(f,x):=\fix(f_{\pi})\subset\pi_1(X,x)$, where $f_{\pi}: \pi_1(X,x)\rightarrow \pi_1(X,x)$ is the induced endomorphism.
Since it is independent of the choice of $x\in \F$, up to isomorphism, the \emph{stabilizer} of a fixed point class $\F$ is defined as
$\stab(f,\F):=\stab(f,x)$, for any $x\in \F$. The \emph{rank} of $\F$ is defined as
$$\rk(f,\F):=\rk\, \stab(f,\F),$$
where the rank of a group is the minimal number of generators.

For a selfmap $f$ of a surface or a graph $X$, they defined the \emph{characteristic} of a fixed point class $\F$ of $f$ as
$$\chr(f,\F):=1-\rk(f,\F),$$
with the only exception that $\chr(f,\F):=\chi(X)=2-\rk(f,\F)$ when $X$ is a closed surface and $\stab(f,\F)=\pi_1(X)$.
What is more, some relations between the characteristic and the index of fixed point classes of graphs and surfaces are proved.

\begin{thm}[Jiang-Wang-Zhang, \cite{JWZ}]\label{JWZ main theorem}
Suppose $X$ is either a connected finite graph or a connected
compact hyperbolic surface, and $f: X\rightarrow X$ is a selfmap.
Then

$\mathrm{(A)}$ $\ind(f, \F)\leq \chr(f, \F)$ for every fixed point class
$\F$ of $f$;

$\mathrm{(B)}$ when $X$ is not a tree,
$$\sum_{\ind(f, \F)+\chr(f, \F)<0}\{\ind(f, \F)+\chr(f, \F)\}\geq 2\chi(X),$$
where the sum is taken over all fixed point classes $\F$ with
$\ind(f, \F)+\chr(f, \F)<0$.
\end{thm}

An analogue on $3$-manifolds can be found in \cite{Z}.

In this paper, we are primarily interested in a $\pi_1$-injective selfmap $f: X\to X$ of a graph $X$, i.e., $f$ induces an injective endomorphism $f_{\pi}$ of the fundamental group $\pi_1(X)$. In this setting, we will introduce a new homotopy invariant $a(f, \F)\in\mathbb{N}$ (see Definition \ref{def. a(f,F)}) for each fixed point class $\F$ of $f$. Thanks to this new invariant, we will improve the inequality $\mathrm{(A)}$ of Theorem \ref{JWZ main theorem} into an inequality (equality when $\chi(X)\geq -1$) among $\ind(f, \F)$, $\rk(f, \F)$ and $a(f,\F)$. For convenience, we define the \emph{improved characteristic} of a fixed point class $\F$ of $f$ as $\ichr(f,\F):=\chr(f,\F)-a(f,\F)$, namely,
$$\ichr(f,\F):=1-\rk(f,\F)-a(f,\F).$$
Note that $\ichr(f,\F)\leq 1$ and only depends on the induced endomorphism $f_{\pi}$ of $\pi_1(X)$.

For brevity, we will write $\ind(\F), \rk(\F), a(\F), \chr(\F)$ and $\ichr(\F)$ if no confusion exists for the selfmap $f$ in the context. In Nielsen fixed point theory, a fixed point class $\F$ is allowed to be empty. In that case the above definitions of $\rk(\F), a(\F)$ and $\ichr(\F)$ do not make sense. Alternative approaches using covering spaces and paths, will be given in Section \ref{Sect. Improved characteristics}.

The main results of this paper are the following.

\begin{thm}\label{main thm 1}
Let $X$ be a connected finite graph and $f: X\to X$ be a $\pi_1$-injective selfmap. Then for every fixed point class $\F$ of $f$, we have
$$\ind(\F)\leq\ichr(\F).$$
\end{thm}

When the Euler characteristic $\chi(X)\geq 0$, the equality $\ind(\F)=\ichr(\F)$ holds immediately, see Example \ref{exam: circle}. When $\chi(X)=-1$, we have

\begin{thm}\label{main thm for figure 8}
If $X$ is a connected finite graph with Euler characteristic $\chi(X)=-1$ and $f: X\to X$ is a $\pi_1$-injective selfmap, then for every essential fixed point class $\F$ of $f$, we have
$$\ind(\F)=\ichr(\F),$$
and for every inessential fixed point class $\F$, we have $0=\ind(\F)\leq \ichr(\F)\leq 1.$
\end{thm}

\begin{conj}\label{conj}
The equality $\ind(\F)=\ichr(\F)$ always holds in Theorem \ref{main thm 1}.
\end{conj}

In other words, we conjecture that $\ind(\F)=1-\rk(\F)-a(\F)$, for some evidence, see Section \ref{sect. examples}. This equality will give a new algebraic approach for computing indices of fixed point classes of $\pi_1$-injective selfmaps of graphs.

\begin{rem}
For any selfmap $f:X\to X$ of a connected finite graph $X$, if $f$ is not $\pi_1$-injective, by \cite[Lemma A]{J2}, it is a \emph{mutant} (see \cite{J2} for a definition) of a $\pi_1$-injective selfmap $g$ of a connected finite graph $Y$ with $\chi(Y)\geq \chi(X)$. Since mutants have the same set of indices of essential fixed point classes, we can also compute the index by $\ichr(\F)$ for an essential fixed point class $\F$ of $f$ if Conjecture \ref{conj} is true. Indeed, let $\F'\in \fpc(g)$ be the corresponding fixed point class of $\F$, then
$$
\ind(f,\F)\xlongequal[]{mutant~ invar.}\ind(g,\F')\xlongequal[]{Conj.~\ref{conj}}\ichr(g,\F').
$$
\end{rem}

Replacing $\ind(\F)$ with $\ichr(\F)$ in the inequality $\mathrm{(B)}$ of Theorem \ref{JWZ main theorem}, as a corollary of Theorem \ref{main thm 1}, we have $2\chi(X)-1\leq \ichr(\F)\leq 1$ and the following bound immediately.

\begin{cor}\label{main cor}
Suppose $X$ is a connected finite graph but not a tree, and $f: X\rightarrow X$ is a $\pi_1$-injective selfmap.
Then
$$\sum_{\F\in\fpc(f)}\max\{0, ~~\rk(\F)+a(\F)/2-1\}\leq -\chi(X),$$
where the sum is taken over all fixed point classes $\F$ of $f$.
\end{cor}

For an injective endomorphism $\phi: F_n\to F_n$ of a free group $F_n$ of rank $n\geq 1$, it induces an endomorphism $\phi^{\mathrm{ab}}$ of the abelianization of $F_n$,
$$\phi^{\mathrm{ab}}: ~\Z^n\to \Z^n.$$
Let $\tr(\phi^{\mathrm{ab}})$ be the trace of a matrix of $\phi^{\mathrm{ab}}$. For any $c\in F_n$, let $i_c: F_n\to F_n, ~g\mapsto cgc^{-1}$ be the inner automorphism induced by $c$. Let $a(\phi)$ be the number of \emph{equivalence classes of attracting fixed points} for the action of $\phi$ on the boundary of $F_n$, see Definition \ref{def. a(phi)}.

Using the trace $\tr(\phi^{\mathrm{ab}})$, we have the following theorem giving an approach to roughly detecting whether fixed words exist or not.

\begin{thm}\label{main thm: existence of fixed words}
Let $\phi$ be an injective endomorphism of a free group $F_n$. Then there exists $c\in F_n$ such that
$$\rk\fix(i_c\comp \phi)=a(i_c\comp \phi)=0$$
if the trace $\tr(\phi^{\mathrm{ab}})<1$; and
$$\rk\fix(i_c\comp \phi)+a(i_c\comp \phi)>1$$
if $n\leq 2$ and $\tr(\phi^{\mathrm{ab}})>1$.
\end{thm}

Moreover, we obtain an upper bound for the rank of the fixed subgroup and the number $a(\phi)$ of an injective endomorphism $\phi$ of a free group.

\begin{thm}\label{main thm 2}
Let $\phi$ be any injective endomorphism of a free group $F_n$. Then
$$\rk\fix(\phi)+a(\phi)/2\leq n.$$
\end{thm}

In \cite{GJLL}, Gaboriau, Jaeger, Levitt and Lustig proved the inequality above  for automorphisms of $F_n$, by using groups acting on $\mathbb{R}$-trees. Our proof for general case is based on Theorem \ref{main thm 1} and Bestvina-Handel's train track maps.

The paper is organized as follows. In Section \ref{Sect. Attracting fixed pts of free gps}, we introduce attracting fixed points and $a(\phi)$ for any injective endomorphism $\phi$ of a free group. In Section \ref{Sect. Improved characteristics}, we first give the background of Nielsen theory and define the improved characteristic for fixed point classes of selfmaps of graphs, and then show some invariance of that. In Section \ref{Sect. gragh maps}, we study the improved characteristic of selfmaps of graphs by an approach using train track maps. In Section \ref{sect. proof of thms}, we complete the proofs of Theorem~\ref{main thm 1}, Theorem~\ref{main thm for figure 8}, Theorem~\ref{main thm: existence of fixed words} and Theorem~\ref{main thm 2}. Finally in Section \ref{sect. examples}, we give some examples supporting Conjecture \ref{conj}.\\

\noindent\textbf{Acknowledgements.}  The authors would like to thank Alexander Fel'shtyn for helpful discussions, and thank Shida Wang for detailed comments.

%============================================================================================================================================================================================

\section{Attracting fixed words and $a(\phi)$ of free groups}\label{Sect. Attracting fixed pts of free gps}

In \cite[Sect. 1]{GJLL}, the authors introduced attracting fixed words and the number $a(\phi)$ of an automorphism $\phi$ of a free group. In this section, we will extend that to injective endomorphisms of free groups.

\subsection{Attracting fixed word and attracting fixed point}\label{subsect. Attracting fixed point}

Let $F$ be a free group of rank $n$. Fixed a basis (i.e., a free generating set) $\Lambda=\{g_1,\ldots, g_n\}$ for $F$, we view $F$ as the set of reduced words in the letters $g_i^{\pm1}$, and $\partial F$ as the set of infinite reduced words $W=w_1w_2\cdots w_i\cdots$, i.e., $w_i\in \Lambda^{\pm}=\{g_1^{\pm1},\ldots, g_n^{\pm1}\}$ and $w_i\neq w_{i+1}^{-1}$. Denote $W_i=w_1\cdots w_i$. The \emph{word length} of a word $W\in F$ with respect to $\Lambda$ is written $|W|$. Given two finite or infinite reduced words $W, V\in \bar F:=F\sqcup \partial F$, let $W\wedge V$ be the longest common initial segment of $W$ and $V$. The \emph{initial segment metric} $d_{i.s}$ on $\bar F$ is defined by $d_{i.s}(W,W)=0$ and for $W \neq V$,
\begin{equation}\nonumber
d_{i.s}(W,V)=\frac{1}{1+|W\wedge V|}.
\end{equation}
With this metric, $\bar F$ is compact (called end completion as in \cite{C} or compactification as a hyperbolic group in the sense of Gromov), and $F$ is dense in $\bar F$.
The boundary $\partial F$ is a compact space that is homeomorphic to a Cantor set when $n\geq 2$. A sequence of reduced words $V_p\in \bar F$ \emph{converges} to an infinite word $W\in \partial F$ if and only if $\lim_{p\to +\infty}|W\wedge V_p|=+\infty$.

The natural actions of $F$ and $\aut(F)$ on $F$ extend continuously to $\bar F$: a left multiply $W: F\to F$ by a word $W\in F$ and an automorphism $f: F\to F$ extend uniquely to homeomorphisms $W: \partial F\to \partial F$ and $\bar f: \partial F\to \partial F$ (see \cite{C}), respectively. Any finitely generated subgroup $F'<F$ is quasi-convex \cite{Sh1}, and hence an inclusion induces a natural embedding $\partial F'\hookrightarrow \partial F$ (\cite[p. 115]{CDP}). For an injective endomorphism $\phi:F\to F$, since $F\cong \phi(F)<F$, we have $\partial F\cong \partial(\phi (F))\hookrightarrow \partial F$. Therefore

\begin{lem}\label{lem phi extend to boudary F}
Let $\phi:F\to F$ be an injective endomorphism of $F$. Then $\phi$ can be extended to a continuous injective map $\bar\phi: \partial F\to \partial F$.
\end{lem}

For now on, let $\phi:F\to F$ be an injective endomorphism of $F$. By \cite[p.32, Lem II.2.4]{DV}, there is a \emph{cancelation bound} $\B>0$ for $\phi$:
$$|\phi(W\cdot V)|\geq |\phi(W)|+|\phi(V)|-2\B,$$
whenever $W,V\in F$ are finite reduced words, and $W\cdot V$ denote the product $WV$ if there is no cancelation between $W$ and $V$, i.e., $|W\cdot V|=|W|+|V|$.

Let $W=w_1\cdots w_i\cdots\in \partial F$ be a fixed infinite reduced word of $\phi$. Write
$$\phi(W_i)=W_{k(i)}\cdot V_i$$
 with $W_{k(i)}=W\wedge \phi(W_i)$ and hence $k(i)=|W\wedge \phi(W_i)|$. Since $W$ is fixed by $\phi$, the sequence $k(i)\to +\infty$ as $i$ increases. Bounded cancelation implies $|V_i|\leq \B$, and $|k(i+1)-k(i)|$ is bounded by the constant $\max\{\phi(g_1),\ldots, \phi(g_n)\}$ depending only on $\phi$. A fixed infinite word $W$ is said to be an \emph{attracting fixed word} of $\phi$ if
 $$\lim_{i\to+\infty} |W\wedge \phi(W_{i})|-i = +\infty.$$
Note that there exists $i_0>0$ such that for all $i\geq i_0$, we have $k(i)\geq i+\B+1$. For any reduced word $W'\in \bar F$, let $W_i=W\wedge W'$ and $W'=W_i\cdot V'$, then $\phi(W')=\phi(W_i)\phi(V')$ and the cancelation length between $\phi(W_i)=W_{k(i)}\cdot V_i$ and $\phi(V')$ is bounded by $\B$. Therefore, if $|W\wedge W'|=i\geq i_0$, then $|W\wedge \phi(W')|\geq k(i)-\B>i=|W\wedge W'|$.
So

\begin{lem}\label{critiaon for attracting fixed word}
Let $W$ be an attracting fixed word of $\phi$. Then there exists an integer $i_0$ such that $|W\wedge \phi(W')|>|W\wedge W'|$ for any word $W'$ with $|W\wedge W'|\geq i_0$.
\end{lem}

We say that a fixed infinite word $W$ is an \emph{attracting fixed point} of $\phi$ if there exists a neighborhood $\mathcal{U}$ of $W\in \bar F$ such that
$$W'\in \mathcal U\Longrightarrow \lim_{p\to +\infty}\phi^p(W')=W.$$
Note that by the definition of the initial segment metric on $\bar F$, $W'\in \mathcal U$ if and only if $|W\wedge W'|>r^{-1}-1$, where $r<1$ is the radium of $\mathcal U$.

By a similar discuss as in the proof of \cite[Proposition 1.1]{GJLL}, we have

\begin{prop}
Let $W\in \partial F$ be a fixed infinite word of $\phi$. If $W$ is an attracting fixed word of $\phi$, then $W\not\in \partial (\fix\phi)$. Moreover, $W$ is an attracting fixed point of $\phi$  if and only if it is an attracting fixed word of $\phi$.
\end{prop}

\begin{proof}
Let $W$ be an attracting fixed word of $\phi$. If $W\in \partial (\fix\phi)$, we can write $W=V_1\cdot V_2\cdots$ and  $W_{s_i}=V_1\cdots V_i$, where the reduced finite words $V_i\in \fix\phi$. Then $\phi(W_{s_i})=W_{s_i}$ and $k(s_i)=s_i$ for all $i$, contradicting the hypothesis that $W$ is attracting. Therefore, $W\not\in \partial (\fix\phi)$.
Moreover, pick $i_0$ as in Lemma \ref{critiaon for attracting fixed word} and $\mathcal U$ a neighborhood with radium $<i_0^{-1}$. For any $W'\in \mathcal U$, $d_{i.s}(W,W')=(1+|W\wedge W'|)^{-1}<i_0^{-1}$, hence $|W\wedge W'|\geq i_0$. By Lemma \ref{critiaon for attracting fixed word}, we have $|W\wedge \phi(W')|>|W\wedge W'|$, and hence $\lim_{p\to +\infty}|W\wedge \phi^p(W')|=+\infty$, namely, $\lim_{p\to +\infty}\phi^p(W')=W$. Therefore, $W$ is an attracting fixed point of $\phi$.

Now let $W$ be an attracting fixed point of $\phi$. Then for $i$ large enough, we have $W=\lim_{p\to+\infty}\phi^p(W_i)$. Consider the words $U_i=W_i^{-1}\phi(W_i)$. Note that
$$U_p=U_i\quad \Longrightarrow\quad W_pW_i^{-1}\in\fix\phi.$$
If the sequence of words $U_p$ takes the same value infinity times, fix $i$ we get
$$W=\lim_{p\to +\infty}W_p=\lim_{p\to +\infty}W_pW_i^{-1}\in\partial(\fix\phi),$$
contradicting the hypothesis that $W$ is an attracting fixed point. Therefore, $\lim_{p\to +\infty}|U_p|=+\infty$. Recall that
$\phi(W_p)=W_{k(p)}\cdot V_p$ with $|V_p|\leq \B$, we have
$$|k(p)-p|+\B\geq |W_p^{-1}W_{k(p)}\cdot V_p|=|W_p^{-1}\phi(W_p)|=|U_p|\to +\infty$$
as $p$ goes to infinity. Since $|k(p+1)-k(p)|$ is bounded, and $W$ is an attracting fixed point of $\phi$, we have $\lim_{p\to+\infty}(k(p)-p)=+\infty$, i.e., $W$ is an attracting fixed word of $\phi$.
\end{proof}

\subsection{The number $a(\phi)$ and similarity invariance}\label{subsect. similarity invariance}

Now let us give the definition of the number $a(\phi)$ of \emph{equivalence classes of attracting fixed points}.

\begin{defn} \label{def. a(phi)}
Let $\phi:F\to F$ be an injective endomorphism of a free group $F$. We say that two fixed infinite words $W,W'\in \partial F$ are \emph{equivalent} if there exists a fixed word $U\in \fix(\phi)$ such that $W'=UW$. Note that any word equivalent to an attracting fixed word of $\phi$ is also an attracting fixed word of $\phi$. Let $\mathscr{A}(\phi)$ be the set of \emph{equivalence classes of attracting fixed words} of $\phi$, and let $a(\phi)$ be the cardinality of $\mathscr{A}(\phi)$.
\end{defn}

\begin{rem}
Let $\mathcal A(\phi)$ be the set of attracting fixed words of $\phi$. Then $\mathscr A(\phi)$ equals $\fix(\phi)\backslash\mathcal A(\phi)$, the set of orbits of $\fix(\phi)$ acting on $\mathcal A(\phi)$.
\end{rem}

For a free group $F$, let $\edo(F)$ (resp. $\inn(F)$, $\inj(F)$) be the set of endomorphisms (resp. inner automorphisms, injective endomorphisms) of $F$.

Consider the natural left action of  $\inn(F)$ on $\edo(F)$. For an endomorphism $\phi\in\edo(F)$, we say that the orbit of $\phi$ is an \emph{Inn-coset}, written $\inn\phi$. Namely, two endomorphisms $\phi_1,\phi_2\in\edo(F)$ are in the same Inn-coset if and only if there exists $m\in F$ such that $\phi_2=i_m\comp \phi_1$ with $i_m(g)=mgm^{-1}$ for any $g\in F$. We say that $\phi_1$ and $\phi_2$ are \emph{similar} if $m$ can be written as $m=c\phi_1(c^{-1})$ for some $c\in F$, or equivalently, if $\phi_2=i_{c}\comp \phi_1\comp (i_c)^{-1}$.

Note that every Inn-coset is a disjoint union of similarity classes, and the rank $\rk\fix(\phi)$ and the number $a(\phi)$ are both similarity invariants.

As an algebraic version of Corollary \ref{main cor}, we will prove the following stronger version of Theorem \ref{main thm 2} in Section \ref{sect. proof of thms}.

\begin{thm}\label{stronger version of main thm 2}
Let $\phi$ be an injective endomorphism of a free group $F_n$ of rank $n\geq 2$. Then
$$\sum_{[\psi]\subset \inn\phi}\max\{0,~~\rk\fix(\psi)+a(\psi)/2-1\}\leq n-1,$$
where the sum runs over all distinct similarity classes $[\psi]$ contained in the Inn-coset $\inn\phi$.
\end{thm}

\begin{exam}\label{exam:rankone}
Let $F$ be the free group of rank $1$, i.e., $F=\langle g\mid - \rangle\cong \mathbb{Z}$. Then any endomorphism $\phi: F\to F$  has the form $\phi(g)=g^k$ with  $\tr(\phi^{\mathrm{ab}})=k$ for some integer $k$.

Note that the boundary $\partial F$ consists of two points: $gg\cdots g\cdots$ and  $g^{-1}g^{-1}\cdots g^{-1}\cdots$. We have that
\begin{center}
\begin{tabular}{|l|c|c|c|c|}
  \hline
  % after \\: \hline or \cline{col1-col2} \cline{col3-col4} ...
 $\tr(\phi^{\mathrm{ab}})$ & $\phi(g)$ & $\fix(\phi)$ & $\rk\fix(\phi)$ & $a(\phi)$ \\\hline
   $0$  & $1$ & $\{1\}$ & $0$ & N/A \\
   $1$  & $g$ & $\mathbb{Z}$ & $1$ & $0$\\
   $k>1$ & $g^k$ & $\{1\}$ & $0$ & $2$\\
   $k<0$ & $g^k$ & $\{1\}$ & $0$ & $0$\\
  \hline
\end{tabular}
\end{center}
The trivial endomorphism is not injective, and therefore $a(\phi)$ is not defined. For the identity endomorphism, each element in $F$  is fixed. It is obvious that the two infinite words are both fixed, but are not attracting.
\end{exam}

%============================================================================================================================================================================================
\section{Improved characteristic of fixed point classes}\label{Sect. Improved characteristics}

In \cite[Section 2]{JWZ}, the authors gave some definitions and facts of fixed point classes of graph selfmaps. In this section, we will state them and define a new homotopy invariant $a(f, \F)\in\mathbb{N}$, called \emph{improved characteristic}, for any fixed class $\F$ of a $\pi_1$-injective selfmap $f:X\rightarrow X$ of a connected finite graph $X$.

\subsection{Fixed point class}\label{defn Fpc by covering}

Let $p: \tilde{X}\to X$ be the universal covering of $X$, with group $\pi$ of covering transformations
which is identified with the fundamental group $\pi_1(X)$.

For any lifting $\tilde{f}:\tilde{X}\to \tilde{X}$ of $f$, the projection of its fixed point set is called a \emph{fixed
point class} of $f$ , written $\F=p(\fix\tilde{f})$. Strictly speaking, we say two liftings $\tilde{f}$ and $\tilde{f'}$
of $f$ are \emph{conjugate} if there exists $\gamma\in \pi$ such that $\tilde{f'}=\gamma^{-1}\comp \tilde{f}\comp \gamma$. Then
$\F=p(\fix\tilde{f})$ is said to be the fixed point class of $f$ \emph{labeled} by the conjugacy class
of $\tilde{f}$. Thus, a fixed point class always carries a label which is a conjugacy class of
liftings. The fixed point set $\fix f$ decomposes into a disjoint union of fixed point
classes. When $\fix\tilde{f}= \emptyset$, we call $\F=p(\fix\tilde{f})$ an \emph{empty} fixed point class.

Empty fixed point classes have the same index $0$ but may have different labels hence be regarded as different. We would better think of them as hidden rather than nonexistent.

\subsection{Improved characteristic}

Each lifting $\tilde{f}$ induces an injective endomorphism $\tilde{f}_{\pi}:\pi\to\pi$ defined by
$$\tilde{f}\comp \gamma=\tilde{f}_{\pi}(\gamma)\comp \tilde{f},\quad \gamma\in \pi.$$
If two liftings $\tilde{f}$ and $\tilde{f'}$ label the same fixed point class $\F$, i.e., there exists $\gamma\in \pi$ such that $\tilde{f'}=\gamma^{-1}\comp \tilde{f}\comp \gamma$, then the induced endomorphism
$\tilde f'_{\pi}=i_{\gamma^{-1}}\comp\tilde f_{\pi}\comp i_{\gamma}\in\inj(\pi)$ and hence the fixed subgroup $\fix(\tilde f_{\pi})\cong \fix (\tilde f'_{\pi})$ and the number $a(\tilde f_{\pi})=a(\tilde f'_{\pi})$.

The \emph{stabilizer} of a fixed point class $\F=p(\fix\tilde{f})$ is defined as the subgroup
$$\stab(f,\F):=\{\gamma\in\pi|\gamma^{-1}\comp\tilde{f}\comp \gamma=\tilde f\},$$
which is identical to the fixed subgroup $\fix(\tilde f_{\pi})=\{\gamma\in\pi|\tilde f_{\pi}(\gamma)=\gamma\}$.
Up to group isomorphism, it is independent of the choice of $\tilde f$ in its conjugacy class.

For nonempty fixed point classes, the definitions above reduces to the simpler ones given in
Section \ref{Sect. introduction}.

\begin{defn}\label{def. a(f,F)}
Let $f:X\to X$ be a $\pi_1$-injective graph map, and $\F=p(\fix \tilde f)$ be any fixed point class of $f$. Define $\rk(f, \F):=\rk\stab(f,\F)$, $a(f, \F):=a(\tilde f_{\pi})$,
and define the \emph{improved characteristic} of $\F$ to be
$$\ichr(f,\F):=1-\rk(f, \F)-a(f,\F).$$
\end{defn}

\begin{rem}
Note that $2\chi(X)-1\leq\ichr(f,\F)\leq 1$ by Corollary \ref{main cor}.
\end{rem}

By the same argument as the proof of Proposition 2.3 in \cite{GJLL}, we have

\begin{prop}\label{empty fpc has ichr 1}
Let $f:X\to X$ be a $\pi_1$-injective selfmap of a graph. Then for each empty fixed point class $\F$ of $f$, we have $0\leq \ichr(\F)\leq 1.$
\end{prop}

\begin{proof}
Let $\F=p(\fix\tilde{f})$ be a fixed point class of $f$. By definition, $\ichr(\F)=1-\rk(\F)-a(\F)\leq 1.$ We will show that $\tilde f$ has a fixed point when $\ichr(\F)<0$.

For the lifting $\tilde f:\tilde X\to \tilde X$, as in \cite[p. 19]{BH}, we have the following:

\begin{fixed point criterion}
Let $x,y\in \tilde X$ be two distinct points with the property that
$\tilde f(x)$ is distinct from $x$ and is not contained in the same connected component of $\tilde X\backslash x$ as $y$, and conversely. Then $\tilde f$ has a fixed point on the path $[x,y]$.
\end{fixed point criterion}

Let every edge of $X$ have length $1$. Then $X$ becomes a metric graph, and equip $\tilde X$ with the lifted metric. For any point $P\in \tilde X$, the map
$$j: \pi\to \tilde X, \quad w\mapsto wP$$
gives a quasi-isometric embedding from the covering transformation group $\pi$ to $\tilde X$. This induces a homeomorphism between $\partial \pi$ and the space $\partial \tilde X$ of ends of $\tilde X$, which is independent of the choice $P$. Moreover, the distance between $\tilde f(j(w))=\tilde f_{\pi}(w)\tilde f(P)$ and $j(\tilde f_{\pi}(w))=\tilde f_{\pi}(w)P$ is bounded by the distance between $P$ and $\tilde f(P)$, independent of $w\in \pi$. It follows that the extension of $\tilde f$ to $\partial \tilde X$ agrees with the extension of $\tilde f_{\pi}$ to $\partial \pi$. Therefore, an attracting fixed point of $\tilde f_{\pi}$ defines an attracting fixed point of $\tilde f$ on $\partial \tilde X$.

If $\ichr(\F)<0$, namely, $\rk\fix(\tilde f_{\pi})+a(\tilde f_{\pi})>1$ by definition. First assume that $\tilde f_{\pi}$ has two distinct (possible equivalent) attracting fixed points $W_1, W_2\in \partial \pi$. Then any two points $x,y\in \tilde X$ sufficiently close to the corresponding attracting fixed points in $\partial \tilde X$ satisfy the hypothesis of the above fixed point criterion, and $\tilde f$ has a fixed point. Such points $W_1, W_2$ exist if $a(\tilde f_{\pi})\geq 2$. They also exist if $a(\tilde f_{\pi})=1$ while $\rk\fix(\tilde f_{\pi})=1$, namely there exist an attracting fixed point $W_1\in \partial \pi$ and a nontrivial element $u\in \fix(\tilde f_{\pi})$, because in this case, we can take $W_2=uW_1$. The only remaining case is when $\rk\fix(\tilde f_{\pi})\geq 2$, then \cite[Lemma 2.1]{BH} applies.
\end{proof}

\subsection{Alternative definitions}\label{subsect. alter def.}

The above definitions of fixed point class and improved characteristic involve covering spaces. Now we state an
alternative approach using paths (introduced in \cite[\S2]{JWZ}), which is convenient for us to prove the homotopy invariance of $\ichr(\F)$.

\begin{defn}\label{defn of fpc by route}
By an $f$-$route$ we mean a homotopy class (rel. endpoints) of path $w:I\rightarrow X$ from a point $x\in X$ to $f(x)$. For brevity, we shall often say the path $w$ (in place of the path class $[ w ]$) is an $f$-route at $x=w(0)$. An $f$-route $w$ gives rise to an endomorphism
$$f_{w}:\pi_1(X,x) \rightarrow\pi_1(X,x),~~[ a] \mapsto [ w(f\comp a)\overline w ] $$
where $a$ is any loop based at $x$, and $\overline w$ denotes the reverse of $w$. For brevity, we will write $f_{\pi}: \pi_1(X)\to\pi_1(X)$
when $w$ and the base point $x$ are omitted. Two $f$-routes $[ w ]$ and $[ w' ]$ are $conjugate$ if there is a path $q:I\rightarrow X$ from $x=w(0)$ to $x'=w'(0)$ such that $[ w' ]=[ \overline qw (f\comp q)]$, that is, $w'$ and $\overline qw (f\comp q)$ homotopic rel. endpoints. We also say that the (possibly tightened) $f$-route $\overline qw (f\comp q)$ is obtained from $w$ by an $f$-$route$ $move$ along the path $q$.
\end{defn}

Note that a constant $f$-route $w$ corresponds to a fixed point $x=w(0)=w(1)\in \fix f$, and the endomorphism $f_{w}$ becomes the usual
$$f_{\pi}:\pi_1(X,x)\rightarrow \pi_1(X,x),~~[ a]\mapsto  [ f\comp a],$$
where $a$ is any loop based at $x$. Two constant $f$-routes are conjugate if and only if the corresponding fixed points can be joined by a Nielsen path. This gives the following definition.

\begin{defn}\label{def. f-route class}
With an $f$-route $w$ (more precisely, with its conjugacy class) we associate a {\em fixed point class} $\F_{w}$ of $f$, which consists of the fixed points that correspond to constant $f$-routes conjugate to $w$. Thus fixed point classes are associated bijectively with conjugacy classes of $f$-routes. A fixed point class $\F_{w}$ can be empty if there is no constant $f$-route conjugate to $w$. Empty fixed point classes are inessential and distinguished by their associated route conjugacy classes.
\end{defn}

Note that this definition is equivalent to the traditional one in Section \ref{defn Fpc by covering} because an
$r$-route specifies a lifting $\tilde f$. If two $f$-route $w, w'$ are associated to the same fixed point class $\F$ of a $\pi_1$-injective graph map $f:X\to X$, then there is a path $q:I\rightarrow X$ from $x=w(0)$ to $x'=w'(0)$ such that $[w']=[\overline qw (f\comp q)]$. It implies a commutative diagram:
$$\xymatrix{
\pi_1(X,x)\ar[d]_{q_\#}^\cong\ar[r]^{f_w} &\pi_1(X,x)\ar[d]_{q_\#}^\cong\\
\pi_1(X,x')\ar[r]^{f_{w'}} & \pi_1(X,x')
}$$
where $q_\#: \pi_1(X,x)\to \pi_1(X,x')$, $[a]\mapsto [\bar qaq]$ is the induced isomorphism by $q$. Therefore, $\rk\fix(f_w)=\rk\fix(f_{w}')$ and $a(f_w)=a(f_{w'})$, and we can define:

\begin{defn}
For a fixed point class $\F_w$ of a $\pi_1$-injective graph map $f:X\to X$, the $stabilizer$ of $\F_w$ is defined to be
$$\stab(f,\F_w):=\fix(f_w)=\{\gamma\in \pi_1(X,w(0))|f_w(\gamma)=\gamma\},$$
which is well defined up to isomorphism. Define $\rk(f, \F_w):=\rk\stab(f,\F_w)$, $a(f, \F_w):=a(f_{w})$,
and define the \emph{improved characteristic} of $\F_w$ to be
$$\ichr(f,\F_w):=1-\rk(f, \F_w)-a(f,\F_w).$$
\end{defn}

\subsection{Invariance}\label{subsect. invariance}

Under a homotopy $H=\{h_t\}_{t\in I}:X\rightarrow X$, each $h_0$-route $w_0$ gives
rise to an $h_1$-route $w_1=w_0\cdot H(w_0(0))$, where $H(w_0(0))$ is the path $\{h_t(w_0(0))\}_{t\in I}$.
Clearly $w_0$ and $w_1$ share the same starting point $w_0(0)$, and
$$(h_0)_{w_0}=(h_1)_{w_1}: \pi_1(X, w_0(0))\to \pi_1(X, w_0(0)).$$
The function $w_0\mapsto w_1$ defines the fixed point class function $\F_{w_0}\mapsto \F_{w_1}$ induced by the homotopy. Therefore, when $h_i$ is $\pi_1$-injective, we have the following fact on the improved characteristic, that is parallel to the one of stabilizer (see \cite[\S2]{JWZ}).

\begin{fact1}[Homotopy invariance]
A homotopy $H=\{h_t\}_{t\in I}:X\rightarrow X$ gives rise to a bijective correspondence $H: \F_{0}\mapsto \F_{1}$ from $\fpc(h_0)$ to $\fpc({h_1})$ with
$$\ind(h_0,\F_{0})=\ind(h_1, \F_{1}),~~
\stab(h_0,\F_{0})\cong \stab(h_1, \F_{1}), ~~a(h_0,\F_{0})=a(h_1,\F_{1}),$$
which indicate that the index $\ind(\F)$ and the improved characteristic $\ichr(\F)$ are homotopy invariants.
\end{fact1}

An \emph{inj-morphism} from a $\pi_1$-injective graph map $f: X\rightarrow X$ to a $\pi_1$-injective graph map $g:Y\rightarrow Y$ means a $\pi_1$-injective map $h:X\rightarrow Y$ such that $h\comp f=g\comp h$. For an $f$-route $w$ at $x=w(0)$, $h$ induces a commutative diagram:
$$\xymatrix{
\pi_1(X,x)\ar[d]_{h_{\pi}}\ar[r]^{f_w} &\pi_1(X,x)\ar[d]^{h_{\pi}}\\
\pi_1(Y,h(x))\ar[r]^{g_{h\comp w}} & \pi_1(Y,h(x))
}$$
where $h_{\pi}:\pi_1(X,x)\to\pi_1(Y,h(x))$ defined by $[a]\mapsto [h\comp a]$ is the injective homomorphism induced by $h$.

Clearly, $h_{\pi}|_{\fix(f_w)}: \fix(f_w)\to \fix(g_{h\comp w})$ is injective. For any two attracting fixed words $W\neq V$ of $f_w$, by Lemma \ref{lem phi extend to boudary F}, $\bar h_{\pi}(W)\neq \bar h_{\pi}(V)$ are two attracting fixed words of $g_{h\comp w}$. So $W\mapsto \bar h_{\pi}(W)$ defines an injection $\bar h_{\pi}: \mathcal A(f_w)\hookrightarrow\mathcal A(g_{h\comp w})$ from the set of attracting fixed words of $f_w$ to that of $g_{h\comp w}$.

\begin{fact2}[Morphism]
An inj-morphism $h$ from a $\pi_1$-injective  selfmap $f$ of graph $X$ to a $\pi_1$-injective selfmap $g$ of graph $Y$ induces a natural function $w\mapsto h\comp w $ from $f$-routes to $g$-routes and a function $\F_{w}\mapsto \F_{h\comp w}$ from $\fpc(f)$ to $\fpc(g)$ such that $h(\F_w)\subset \F_{h\comp w}$, and two injections
$$h_{\pi}:\stab(f,\F_{w})\hookrightarrow\stab(g,\F_{h\comp w}), \quad \bar h_{\pi}:\mathcal A(f_{w})\hookrightarrow \mathcal A(g_{h\comp w}).$$
\end{fact2}

Furthermore, when $h_{\pi}|_{\fix(f_w)}: \fix(f_w)\cong \fix(g_{h\comp w})$, two attracting fixed words $W, V\in \mathcal A(f_w) $ are equivalent if and only if their images $\bar h_{\pi}(W), \bar h_{\pi}(V)\in \mathcal A(g_{h\comp w})$ are equivalent, namely, there exists a word $\omega\in \fix(g_{h\comp w})$ such that $\bar h_{\pi}(W)=\omega \bar h_{\pi}(V)$ if and only if $W=h_{\pi}^{-1}(\omega)V$. It implies an injection $\bar h_{\pi}: \mathscr A(f_w)\hookrightarrow\mathscr A(g_{h\comp w})$ defined by $[W]\mapsto [\bar h_{\pi}(W)]$, where $[W]$ denotes the equivalence class represented by $W$. Therefore

\begin{fact3}[Commutation invariance]
Suppose $\phi: X\to Y$ and $\psi: Y\to X$ are $\pi_1$-injective graph maps. Then the $\pi_1$-injective selfmap  $\psi\comp\phi: X\to X$ and $\phi\comp\psi: Y\to Y$ are said to differ by a $\mathrm{commutation}$. The inj-morphism $\phi$
between them sets up a natural bijective correspondence $\F_{\psi\comp\phi}\mapsto\F_{\phi\comp\psi}$ from $\fpc(\psi\comp\phi)$ to $\fpc(\phi\comp \psi)$, with
$$\ind(\psi\comp\phi,\F_{\psi\comp\phi})=\ind(\phi\comp \psi, \F_{\phi\comp \psi}),$$
and
$$\stab(\psi\comp\phi,\F_{\psi\comp\phi})\cong \stab(\phi\comp \psi, \F_{\phi\comp \psi}), \quad a(\psi\comp\phi,\F_{\psi\comp\phi})=a(\phi\comp \psi,\F_{\phi\comp \psi}),$$
which indicate that $\ind(\F)$ and $\ichr(\F)$ are commutation invariants.
\end{fact3}

\begin{rem}
A homotopy may create nonempty fixed point classes, or remove fixed point
classes. The above correspondence is bijective only when empty fixed point classes are
taken into account.
\end{rem}

\begin{exam}\label{exam: circle}
Let $f: S^1\to S^1$ be a $\pi_1$-injective selfmap of the circle. Then $f$ can be homotoped to a map $g: S^1\to S^1$ defined by $g(e^{i\theta})=e^{ik\theta}$, where $k\neq 0$ is the degree of $f$. Note that all fixed point classes of $g$ induce the same homomorphism $g_{\pi}(z)=z^k$ where $z$ is a generator of the fundamental group $\pi_1(S^1)=\Z$.

If $k=1$, then $g$ is the identity and every fixed point class $\F$ of $g$ is inessential, i.e., $\ind(\F)=0$. From Example \ref{exam:rankone}, we have that $\rk(\F)=1$, $a(\F)=0$ and hence $$\ind(\F)=\ichr(\F)=1-\rk(\F)-a(\F)=0.$$

If $k\neq 1$, then every lifting of $g$ has a fixed point and hence every fixed point class $\F$ of $g$ consists of a single point. It is well-known that all these fixed point classes have the same index $\ind(\F)=\mathrm{sgn}(1-k)$. From Example~\ref{exam:rankone}, we also have $\ind(\F)=\ichr(\F)$.

So, by the homotopy invariance, we have $$\ind(\F)=\ichr(\F)$$
 for every fixed point class $\F$ of $f: S^1\to S^1$.
\end{exam}

%============================================================================================================================================================================================
\section{Graph maps}\label{Sect. gragh maps}

In this section, we will study properties of the improved characteristic of graph maps. Since our discussion of graph maps is based on Bestvina and Handel's theory of train tracks \cite{BH}. We follow their terminology first.

\subsection{Relative train track map}

A \emph{graph} $X$ is a $1$-dimensional (or possibly $0$-dimensional) finite cellular complex.
The $0$-cells and (open) $1$-cells are called \emph{vertices} and \emph{edges} respectively. A \emph{graph
map} $f:X \to Y$ is a cellular map, that is, it maps vertices to vertices. Up to homotopy
there is no loss to assume that the restriction of $f$ to every edge $e$ of X is either locally
injective or equal to a constant map. A graph map $f:X \to Y$ is $\pi_1$-\emph{injective} if it
induces an injective homomorphism of the fundamental group on each component of
$X$. It is an \emph{immersion} if it sends edges to edges and it is locally injective at vertices.
Clearly immersions are always $\pi_1$-injective.

A \emph{path} $p$ in a graph $X$ is a map $p:[0,1]\to X$ that is either locally injective or equal
to a constant map, in the latter case we say that $p$ is a \emph{trivial path}. Combinatorially, a nontrivial path consists of a finite sequence $e_1,e_2,\ldots, e_k$ of oriented edges $e_i$ with consecutive edges $e_i,e_{i+1}$ adjacent but $e_{i+1}$ not the inverse $\bar e_i$ of $e_i$. For a nontrivial
path $p$ in $X$ , its \emph{initial tip} is the maximal initial open subpath that lies in an edge of
$X$ . The \emph{terminal tip} is defined similarly.

A \emph{ray} $\rho$ in a connected graph $X$ (equipping each edge with length $1$) is a locally isometric map $\rho:[0,+\infty)\to X$, or combinatorially, a ray consists of an infinite sequence $e_1,e_2,\ldots$ of oriented edges $e_i$ with consecutive edges $e_i,e_{i+1}$ adjacent but $e_{i+1}$ not the inverse of $e_i$. The point $\rho(0)$ is called the \emph{origin} of $\rho$.
If fixing a universal covering $q:\tilde X\to X$ and $\tilde v\in q^{-1}(\rho(0))$, then a lifting $\tilde \rho: [0,+\infty)\to \tilde X$ with origin ${\tilde\rho(0)=\tilde v}$ is an isometric embedding, and its image looks like a real ray in the tree $\tilde X$. A \emph{sub-ray} of a ray $\tilde \rho$ means a restriction $\tilde \rho|_{[n,+\infty]}$. For any two points $\tilde x,\tilde y\in \tilde X$, let $[\tilde x,\tilde y]$ denote the unique path from $\tilde x$ to $\tilde y$ in the tree $\tilde X$. Two rays (with different origins) are \emph{equivalent} if their intersection has infinite length. In particular, a ray is always equivalent to its sub-rays. Each equivalence class of rays in $\tilde X$ defines an \emph{end} of the tree $\tilde X$. Let $\partial \tilde X$ denote the space of all the ends of $\tilde X$.
Note that there is a bijective correspondence between the space $\partial \tilde X$ and the set of rays in $\tilde X$ with the same origin. For brevity, we often identify a ray or a path with its image.

A \emph{turn} in $X$ is an unordered pair of oriented edges of $X$ originating at a common vertex. A turn is \emph{nondegenerate} if it
is defined by distinct oriented edges, and is \emph{degenerate} otherwise. A graph map $f: X \to Y$ induces a function $Df$ on the set $\E(X)$ of oriented edges of $X$ by sending
an oriented edge $e$ to the first oriented edge of $f(e)$; if $f(e)$ is trivial we say
$Df(e)=0$. A turn $\{e_1,e_2\}$ is \emph{illegal} with respect to $f:X\to Y$ if the image $\{Df^k(e_1), Df^k(e_2)\}$ under some iterate of $Df$ is degenerate, and is \emph{legal} otherwise.

For a selfmap $f:X\to X$ of graph $X$, an $f$-\emph{Nielsen path} is a nontrivial path $p$ in $X$ joining
two fixed points of $f$ such that $f\comp p¡¦\simeq p$ rel. endpoints; it is \emph{indivisible} if it cannot be
written as a concatenation $p=p_1\cdot p_2$, where $p_1$ and $p_2$ are subpaths of $p$ that are
$f$-Nielsen paths.

Add \cite[Lemma 5.11]{BH} or \cite[Proposition IV 3.2]{DV} to \cite[Theorem BH]{JWZ}, we have the following theorem summarizing the results of Bestvina and Handel \cite{BH} that we need.

\begin{thm}\cite[Theorem BH]{JWZ}\label{Thm BH}
Let $X$ be a connected graph but not a tree, and let $f: X\to X$ be a
$\pi_1$-injective map. Then $f$ has the same homotopy type as a graph selfmap $\beta:Z\to Z$,
where $Z$ is a connected graph without vertices of valence $1$ and all fixed points of $\beta$
are vertices, and there is a $\beta$-invariant proper subgraph $Z_0$, containing all vertices of
$Z$. The map $\beta: (Z,Z_0)\to (Z,Z_0)$ of the pair is of one of the following types.

$\mathrm{Type~1}:$ $\beta$ sends $Z$ into $Z_0$.

$\mathrm{Type~2}:$ $\beta$ cyclically permutes the edges in $Z \backslash Z_0$.

$\mathrm{Type~3}:$ $\beta$ expands edges of $Z \backslash Z_0$ by a factor $\lambda > 1$ with respect to a suitable
non-negative metric $L$ supported on $Z \backslash Z_0$, and has the properties (a)-(c) below.

$\mathrm{(a)}$ For every oriented edge $e$ in $Z \backslash Z_0$, $D\beta(e)$ lies in $Z \backslash Z_0$.

$\mathrm{(b)}$ There is at most one indivisible $\beta$-Nielsen path that intersects $Z \backslash Z_0$.

$\mathrm{(c)}$ If $p$ is an indivisible $\beta$-Nielsen path that intersects $Z \backslash Z_0$, then $p=p_1\bar p_2$, where $p_1, p_2$ are $\beta$-legal paths with length $L(p_1)=L(p_2)$, and the turn $\{\bar p_1, \bar p_2\}$ is the unique illegal turn in $Z \backslash Z_0$ (at a vertex $v_p=p_1(1)=p_2(1)$ of valence $\geq 3$ in $Z$) which degenerates under $D\beta$. Moreover, $\beta(p_i)=p_it$ $(i=1,2)$ where $t$ is a $\beta$-legal path (see Figure~\ref{fig:indivisible Nielsen path} below).
\begin{figure}[h]
\begin{center}
\setlength{\unitlength}{0.7mm}
\begin{picture}(60,50)(-30,-25)
\put(0,0){\line(1,0){20}}
\put(-20,10){\line(2,-1){20}}
\put(-20,-10){\line(2,1){20}}

\put(0,0){\vector(1,0){10}}
\put(-20,10){\vector(2,-1){10}}
\put(-20,-10){\vector(2,1){10}}

\put(0,2){\makebox(0,0)[cb]{$  v_p$}}
\put(10,2){\makebox(0,0)[cb]{$  t $}}
\put(-10,8){\makebox(0,0)[cb]{$ p_1$}}
\put(-10,-8){\makebox(0,0)[ct]{$ p_2$}}
\end{picture}
\end{center}\caption{An indivisible Nielsen path.}\label{fig:indivisible Nielsen path}
\end{figure}
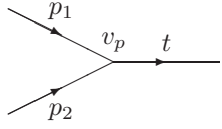
\end{thm}

A graph map as in Theorem \ref{Thm BH} is called a \emph{relative train track} map, for short, say an RTT map.

Let $\beta:(Z,Z_0)\to (Z,Z_0)$ be an RTT map. Denote $\beta_0:=\beta|_{Z_0}: Z_0\to Z_0$. In order to discuss the index and the improved characteristic, we introduce some notations.

Let $\V(Z)$ be the set of vertices of $Z$, and $\E(Z \backslash Z_0)$ the set of oriented edges of $Z \backslash Z_0$. For a fixed vertex $v\in \V(Z)$,
let $\delta(v)$ be the number of oriented edges $e\in \E(Z \backslash Z_0)$ with the additional requirement that $e$ gets initially expanded
along itself by $\beta$ (In Type 2, when $Z \backslash Z_0$ is a single edge $e$, consider that $e$ gets initially expanded
along itself on one tip and shrinks on the other), that is
$$\Delta(v):=\{e\in \E(Z \backslash Z_0)|e(0)=v, ~~D\beta(e)=e\}, \quad \delta(v):=\#\Delta(v).$$
For a nonempty fixed point class $\F$ of $\beta$, let
$$\Delta(\F):=\{e\in \E(Z \backslash Z_0)|e(0)\in \F, ~~D\beta(e)=e\}=\bigsqcup_{v\in\F}\Delta(v),$$
and let $\delta(\F):=\#\Delta(\F)=\sum_{v\in\F}\delta(v).$
Recall that $\ind(\beta, v)=\ind(\beta_0, v)-\delta(v)$ for any fixed point $v$ of $\beta_0$. So, by the additivity of index, for every $\beta$-fixed point class $\F$, we have
\begin{equation}\label{equa. additivity of index}
\ind(\beta, \F)=\ind(\beta_0, \F)-\delta(\F),
\end{equation}
where $\ind(\beta_0, \F)=\sum_{i=1}^k\ind(\beta_0, \F_i)$ if $\F=\sqcup \F_i$ is a union of finite many $\beta_0$-fixed point classes $\F_i, i=1,\ldots,k$.

\subsection{A bijective correspondence}\label{subsect. bijective correspondence}
For any nonempty fixed point class $\F$ of $\beta: Z\to Z$, to give some information on $a(\beta,\F)$, we fix a universal covering $q: \tilde Z\to \tilde Z$ of $Z$,  with group $\pi$ of covering transformations identified with the fundamental group $\pi_1(Z, v_0)$ for $v_0\in \F$.
Pick $\tilde v_0\in q^{-1}(v_0)$ and a lifting $\tilde \beta:\tilde Z\to \tilde Z$ of $\beta$ with $\tilde \beta(\tilde v_0)=\tilde v_0$, then $\F=q(\fix\tilde \beta)$, and the lifting $\tilde{\beta}$ induces an injective endomorphism $\beta_{\pi}:\pi\to\pi$ defined by
\begin{equation}
\tilde{\beta}\comp \gamma=\beta_{\pi}(\gamma)\comp \tilde{\beta}, \quad \forall \gamma\in \pi.
\end{equation}

Endow $\tilde Z$ with a metric $d$ with each edge length $1$. Then the map
$$j: \pi\to \tilde Z, \quad \gamma\mapsto \gamma(\tilde v_0)$$
is $\pi$-equivariant (i.e. $\alpha(j(\gamma))=j(\alpha\gamma)$ for any $\alpha, \gamma\in \pi$), and gives a quasi-isometric embedding from the covering transformation group $\pi$ to the covering space $\tilde Z$. This induces a $\pi$-equivariant homeomorphism $\bar j: \partial \pi\to \partial \tilde Z$ between $\partial \pi$ and the space $\partial \tilde Z$ of ends of $\tilde Z$. Moreover, since $\tilde \beta(j(\gamma))=\beta_{\pi}(\gamma)(\tilde \beta(\tilde v_0))=\beta_{\pi}(\gamma)(\tilde v_0)=j(\beta_{\pi}(\gamma))$, we have a commutative diagram
$$\xymatrix{
\pi\ar[d]_j\ar[r]^{\beta_{\pi}} & \pi\ar[d]^{j}\\
\tilde Z\ar[r]^{\tilde \beta}        & \tilde Z
}$$
It follows that the extension of $\tilde \beta$ to $\partial \tilde Z$ agrees with the extension of $\beta_{\pi}$ to $\partial \pi$. Therefore, we have

\noindent\textbf{Assertion $\star$}. \emph{An attracting fixed word $W\in \partial \pi$ of $\beta_{\pi}$ defines an attracting fixed point $\bar j(W)\in\partial \tilde Z$ of $\tilde \beta$, and the $\pi$-equivariant homeomorphism $\bar j: \partial \pi\to \partial \tilde Z$ induces a bijective correspondence
$$\bar j|_{\mathcal{A(\beta_{\pi})}}: \mathcal{A(\beta_{\pi})}\to \mathcal{A(\tilde \beta)}$$
between the set $\mathcal{A(\beta_{\pi})}$ of attracting fixed words of $\beta_{\pi}$ in $\partial \pi$ and the set $\mathcal{A(\tilde \beta})$ of attracting fixed points of $\tilde \beta$ in $\partial \tilde Z$.}

Here a fixed end $\mathcal{E}\in \mathcal A(\tilde \beta)$ represented by a ray $\tilde \rho=\tilde e_1\cdots\tilde e_i\cdots\subset\tilde Z$ is an \emph{attracting fixed point} of $\tilde \beta$, if there exists a number $N>0$ such that for any point $\tilde x\in \tilde Z$, we have
$$d([\tilde v_0, \tilde x]\cap \tilde \rho)>N \Longrightarrow \lim_{k\to+\infty}d([\tilde v_0, \tilde\beta^k(\tilde x)]\cap \tilde\rho)=+\infty,$$
that is the same as the one in free groups (see Section \ref{subsect. Attracting fixed point}).

\subsection{RTT map of Type 3}\label{subsect. RTT type 3}

In this subsection, let $\beta:(Z,Z_0)\to (Z,Z_0)$ be an RTT map of Type 3 in Theorem \ref{Thm BH}.

A path with no illegal turns in $Z \backslash Z_0$ is said to be \emph{$\beta$-legal}. Under the non-negative metric $L$ supported on $Z \backslash Z_0$, we have $L(\beta(\sigma))=\lambda L(\sigma)$ for any $\beta$-legal path $\sigma$. For convenience,  we state \cite[Lemma 5.8]{BH} as follows:

\begin{lem}\label{key lemma: legal path expending}
Suppose $\sigma=\mu_1\tau_1\mu_2\ldots \tau_{l-1}\mu_l$ is a decomposition of a $\beta$-legal path $\sigma$ into subpaths $\mu_j\subset Z \backslash Z_0$ and $\tau_j\subset Z_0.$  Then
$$[\beta(\sigma)]=\beta(\mu_1)\cdot [\beta(\tau_1)]\cdot \beta(\mu_2)\cdots[\beta(\tau_{l-1})]\cdot\beta(\mu_l)$$
and is $\beta$-legal. Here $[\beta(\sigma)]$ denotes the path tightened from $\beta(\sigma)$, and $\beta(\mu_i)\cdot [\beta(\tau_i)]$ indicates that the turns between $\beta(\mu_i)$ and $[\beta(\tau_i)]$ are legal.
\end{lem}

In the following, we shall give a key lemma that summarizes the relation between $\beta$-invariant edges and attracting fixed words of $\beta_{\pi}$.

\begin{lem}\label{key lemma: expending edges are infinite words}
Let $\beta:(Z,Z_0)\to (Z,Z_0)$ be an RTT map of Type 3, and $\F$ a nonempty fixed class of $\beta$. Suppose $v_0\in \F\subset\V(Z)$ is a fixed point of $\beta$, and $\beta_{\pi}: \pi_1(Z, v_0)\to  \pi_1(Z, v_0)$ is the induced injective endomorphism of $\beta$. Then

$\mathrm (1)$ Every oriented edge $e\in \Delta(\F)$ defines an equivalence class $\mathscr W_e$ of attracting fixed words of $\beta_{\pi}$, and $\mathscr W_e$ does not contain any attracting fixed word of $(\beta_0)_{\pi}: \pi_1(Z_0, v_0)\to  \pi_1(Z_0, v_0)$.

$\mathrm (2)$ Suppose $e_i\in \Delta(\F)~(i=1,2)$ are two edges with initial points $e_i(0)$ two (possibly the same) fixed points in $\F$. Then $\mathscr W_{e_1}$ and $\mathscr W_{e_2}$ are equal if and only if there exists an indivisible Nielsen path $p=p_1\bar p_2$ as in $\mathrm{Type~3(c)}$ of Theorem \ref{Thm BH} such that $e_i=D\beta(p_i)$ are the initial edges of the $\beta$-legal paths $p_i$ for $i=1,2$.

$\mathrm (3)$ For every equivalence class $\mathscr W$ of attracting fixed words of $\beta_{\pi}$ not containing an attracting fixed word of $(\beta_0)_{\pi}$, there exists an oriented edge $e\in \Delta(\F)$ such that $\mathscr W=\mathscr W_e$.
\end{lem}

\begin{proof}
Pick a universal covering $q:\tilde Z\to \tilde Z$ and a lifting $\tilde\beta:\tilde Z\to \tilde Z$ as in Section \ref{subsect. bijective correspondence}. Let $\tilde Z_0:=q^{-1}(Z_0)$.

(1) For any edge $e\in \Delta(\F)$, let $v:=e(0)\in\F$. Recall that there is a non-negative metric $L$ supported on $Z \backslash Z_0$, such that $\beta$ expands $e$ by a factor $\lambda > 1$ with respect to $L$, namely, $L(\beta(e))=\lambda L(e)$, which implies that $\beta(e)$ passes at least two edges of $Z \backslash Z_0$. Property (a) in Theorem \ref{Thm BH} implies both tips of
$\beta(e)$ are in $Z \backslash Z_0$. Furthermore, we have a decomposition of $\beta(e)$ as in Lemma \ref{key lemma: legal path expending}, $\beta(e)=e\cdot\tau_1\cdot\mu_2\cdots \tau_{l-1}\cdot\mu_l$ is a $\beta$-legal path, and
\begin{eqnarray}\label{equaltiy: lifting defines endo on group}
[\beta^2(e)]&=&\beta(e)\cdot[\beta(\tau_1)]\cdot\beta(\mu_2)\cdots [\beta(\tau_{l-1})]\cdot\beta(\mu_l)\nonumber\\
&=&e\cdot\tau_1\cdot\mu_2\cdots \tau_{l-1}\cdot\mu_l\cdot[\beta(\tau_1)]\cdot\beta(\mu_2)\cdots [\beta(\tau_{l-1})]\cdot\beta(\mu_l).\nonumber
\end{eqnarray}
By induction, the common initial segment $[\beta^k(e)]\cap [\beta^{k+1}(e)]=[\beta^k(e)]$ for any $k\in \mathbb{N}$. Therefore, we can define $\beta^{\infty}(e):=\lim_{k\to +\infty}[\beta^k(e)]$. Note that $\beta(\beta^{\infty}(e))=\beta^{\infty}(e)$, so it is a $\beta$-invariant ray in $Z$.

Below we will lift $\beta^{\infty}(e)$ to the universal covering $q:\tilde Z\to \tilde Z$, and then show that the lifting is an attracting fixed point of $\tilde\beta$.

Recall that $\tilde\beta:\tilde Z\to \tilde Z$ is a lifting of $\beta$ such that $\tilde \beta(\tilde v_0)=\tilde v_0$. Since $v$ and $v_0$ are in the same fixed point class $\F$, we have $q^{-1}(v)\cap \fix\tilde \beta\neq \emptyset$. Pick a lifting $\tilde v\in q^{-1}(v)\cap \fix\tilde\beta$ and a lifting $\tilde e$ of $e$ with $\tilde e(0)=\tilde v$.  Then we have a $\tilde \beta$-invariant ray
$$\tilde \beta^{\infty}(\tilde e):=\lim_{k\to +\infty}[\tilde\beta^k(\tilde e)]=\bigcup_{k\to +\infty}[\tilde\beta^k(\tilde e)]: ~~ [0,+\infty]\to \tilde Z$$
which is a lifting of $\beta^{\infty}(e)$, with origin $\tilde v$.

Endow $\tilde Z$ with the lifted metric of $L$ (still written $L$), we have $L([\tilde\beta^k(\tilde e)])=\lambda^kL(\tilde e)$. If we equip $\tilde Z$ with the natural metric $d$ with each edge length $1$, then we have
$$d([\tilde\beta^k(\tilde e)])\geq L([\tilde\beta^k(\tilde e)])/l_0=\lambda^kL(\tilde e)/l_0$$
where $l_0:=\max\{L(e)|e\in \E(Z)\}$ is finite because $Z$ is a finite graph. So the ``bounded cancellation
lemma" (\cite{C} or \cite{DV}) implies that there exists $N>0$ such that for any point $\tilde x\in \tilde Z$,
$$d([\tilde v_0, \tilde x]\cap \tilde \beta^{\infty}(\tilde e))>N \Longrightarrow \lim_{k\to+\infty}d([\tilde v_0, \tilde\beta^k(\tilde x)]\cap \tilde \beta^{\infty}(\tilde e))=+\infty,$$
i.e., the end represented by the $\tilde\beta$-invariant ray $\tilde \beta^{\infty}(\tilde e)$ is an attracting fixed point of $\tilde \beta$ on the space $\partial \tilde Z$. By Assertion $\star$, it defines an attracting fixed word $W_{\tilde e}$ of $\beta_{\pi}$, that is, $\bar j(W_{\tilde e})=\tilde \beta^{\infty}(\tilde e)$.

If $\tilde e'$ is another lifting of $e$ with origin $\tilde v'\in q^{-1}(v)\cap \fix\tilde\beta$, there exists $\omega\in \pi$ such that $\omega(\tilde v)=\tilde v'$. It implies  $\omega\in\fix(\beta_{\pi})$ and the ray $\tilde \beta^{\infty}(\tilde e')=\omega(\tilde \beta^{\infty}(\tilde e))$. So the attracting fixed point $W_{\tilde e'}=\omega W_{\tilde e}$, namely, $W_{\tilde e'}$ and $W_{\tilde e}$ are equivalent and hence represent the same equivalence class of attracting fixed points of $\beta_{\pi}$, which is denoted by $\mathscr W_e$.

If $U$ is an attracting fixed word of $(\beta_0)_{\pi}$, then the ray $\bar j(U)\subset q^{-1}(Z_0)$ and has length $L(\bar j(U))=0$, but any ray representing $\mathscr W_e$ has length $L=\infty$.
So the equivalence class $\mathscr W_e$ can not contain any attracting fixed word of $(\beta_0)_{\pi}$.\\

(2) For two edges $e_i\in \Delta(\F)~(i=1,2)$ with initial points $e_i(0)$ two (possibly the same) fixed points in $\F$, if there exists an indivisible Nielsen path $p=p_1\bar p_2$ as in $\mathrm{Type~3(c)}$ of Theorem \ref{Thm BH} such that $e_i=D\beta(p_i)$ for $i=1,2$, by lifting $p$ to the universal covering $\tilde Z$, we have an indivisible Nielsen path $\tilde p=\tilde p_1\tilde{\bar{p}}_2$ of $\tilde \beta$ with tips $D\tilde \beta(\tilde p_i)=\tilde e_i$ that is a lifting of $e_i$. Moreover, since $\tilde\beta(\tilde p_i)=\tilde p_i\tilde t$ for $\tilde t$ a $\tilde\beta$-legal path, the intersection of the rays $\tilde \beta^{\infty}(\tilde e_1)=\tilde\beta^{\infty}(\tilde p_1)$ and $\tilde \beta^{\infty}(\tilde e_2)=\tilde\beta^{\infty}(\tilde p_2)$ have infinite length, which implies that they represent the same end of $\tilde Z$ and hence the same attracting fixed point of $\tilde\beta$. Therefore, by the above Assertion $\star$, $e_1$ and $e_2$ define the same equivalence class $\mathscr W_{e_1}=\mathscr W_{e_2}$ of attracting fixed words of $\beta_{\pi}$.

Conversely, if the two equivalence classes $\mathscr W_{e_1}$ and $\mathscr W_{e_2}$ of attracting fixed words of $\beta_{\pi}$ are the same, there exists lifting $\tilde e_i$ of $e_i$ such that $W_{\tilde e_1}=W_{\tilde e_2}$. It implies that the intersection $\tilde\rho:=\tilde \beta^{\infty} (\tilde e_1)\cap \tilde \beta^{\infty} (\tilde e_2)$ of the two $\tilde\beta$-invariant rays $\tilde \beta^{\infty} (\tilde e_i)$ has length $L(\tilde \rho)=\infty$. So we can denote $\tilde \beta^{\infty} (\tilde e_i)=\tilde p_i\tilde \rho$, where $\tilde p_i$ is a $\tilde\beta$-legal path with initial edge $\tilde e_i$, and the turn $\{\bar{\tilde p}_1, \bar {\tilde p}_2\}$ is nondegenerate, but the turn $\{D\tilde\beta(\bar{\tilde p}_1), D\tilde\beta(\bar {\tilde p}_2)\}$ is degenerate. So $\tilde p_1\bar{\tilde p}_2$ equals the tightened path $[\tilde \beta(\tilde p_1\bar{\tilde p}_2)]$ and hence is a Nielsen path of $\tilde\beta$. Furthermore, since $\tilde \beta$ expands $\tilde p_i$ by the factor $\lambda>1$ w.r.t the metric $L$, the initial point $\tilde p_i(0)$ is the unique fixed point in $\tilde p_i$. It follows that the Nielsen path $\tilde p_1\bar{\tilde p}_2$ is indivisible, and hence $p_1\bar p_2:=q(\tilde p_1\bar{\tilde p}_2)$ is an indivisible Nielsen path satisfying the property in claim (2).\\

(3) For any equivalence class $\mathscr W$ of attracting fixed words with property as in claim (3),
by the above Assertion $\star$, there exists $W\in \mathscr W$ such that $\bar j(W)\in \partial \tilde Z$ is an attracting fixed point of $\tilde \beta$. Suppose $\rho: [0,+\infty]\to \tilde Z$ is a ray representing $\bar j(W)$ with origin $\rho(0)=\tilde v_0$ (see Figure \ref{fig: ray} below). Since $\bar j(W)$ is attracting (under the metric $d$ on $\tilde Z$ that every edge has length $1$),  the set $S_{\rho}:=\rho\cap\fix\tilde \beta$ of $\tilde\beta$-fixed points in $\rho$ is finite. Indeed, if $S_{\rho}$ is infinite, there exists an infinite sequence $\tilde v_i\in S_{\rho}$ such that the length $d([\tilde v_0,\tilde v_i])$ of $\tilde\beta$-Nielsen path goes to infinity as $i\to +\infty$, but for any given $\tilde v_i\in S_{\rho}$, $d([\tilde v_0,\tilde\beta^k(\tilde v_i)]\cap\rho)=d([\tilde v_0,\tilde v_i])$ does not go to infinity as $k\to +\infty$, contradicting that $\rho$ is attracting.

Let $\tilde v'$ be the last one in $S_{\rho}$, that is
$$d([\tilde v_0,\tilde v'])=\max\{d([\tilde v_0,\tilde v_i])\mid\tilde v_i\in S_{\rho}\}.$$
Let $\rho'$ denote the sub-ray of $\rho$ with fixed origin $\rho'(0)=\tilde v'$, and let $\tilde e$ be the initial oriented edge of $\rho'$ with $\tilde e(0)=\tilde v'$. Note that $\rho'$ also represents the equivalence class $\mathscr W$ of attracting fixed points of $\tilde \beta$. In the tree $\tilde Z$, if $D\tilde\beta(\tilde e)\neq \tilde e$, then any two points $x,y\in \rho'$ sufficiently close to $\tilde v'$ and $\partial \tilde Z$ respectively satisfy the hypothesis of the fixed point criterion (see the proof of Proposition \ref{empty fpc has ichr 1}), and hence $\tilde \beta$ has a fixed point in the path $[x,y]\subset \rho'$, that is, $\tilde \beta$ has another fixed point but not $\tilde v'$ in $\rho'$, contradicting with the choice of $\tilde v'$. Therefore, we have $D\tilde\beta(\tilde e)=\tilde e$ and $[\tilde \beta(\rho')]=\rho'$.

Below we will show that $\tilde e\subset \tilde Z\backslash\tilde Z_0$.

Recall that the equivalence class $\mathscr W$ represented by $\rho'$ does not contain any fixed word in $\partial \pi_1(Z_0)$, then $\rho'$ contains at least one edge in $\tilde Z\backslash\tilde Z_0$. Let $\tilde e'\in\E(\tilde Z\backslash\tilde Z_0)$ be the first one along $\rho'$ from its origin $\tilde v'$, that is
 $$d([\tilde v', \tilde e'(0)]):=\min\{d([\tilde v', \tilde e_i(0)])\mid\tilde e_i\in \rho'\cap\E(\tilde Z\backslash\tilde Z_0)\}.$$
It follows that the (possibly trivial) path $\rho_1:=[\tilde v', \tilde e'(0)]\subset \tilde Z_0$, and $\tilde e'$ is not contained in $\tilde\beta(\rho_1)\subset\tilde Z_0$.

Now we claim that $\rho_1$ is trivial. Otherwise, if $\rho_1$ is nontrivial, namely, $\tilde e\subset \tilde Z_0$ and $\tilde e\neq \tilde e'$. Write $\rho'=\rho_1\rho_2$ for $\rho_2=\rho'\backslash\rho_1$ the sub-ray of $\rho'$ with origin $\rho_2(0)=\tilde e'(0)$. Since $\tilde v'$ is the only fixed point in $\rho'$ and $\tilde\beta(\rho_1)$ does not contain the edge $\tilde e'$, we have $\tilde\beta(\rho_2)$ contains $\rho_2$ as a proper sub-ray, that is, $\tilde\beta(\tilde e'(0))$ is not in the same connected component of $\tilde Z\backslash \tilde e'(0)$ as $\tilde e'(1)$. Then $\tilde e'(0)$ and any point $y\in \rho'$ sufficiently close to $\partial \tilde Z$ satisfy the hypothesis of the fixed point criterion, and $\tilde \beta$ has a fixed point in the path $[\tilde e'(0),y]\subset\rho'$. It contradicts that $\tilde v'$ is the unique fixed point in $\rho'$.

Therefore, $\rho_1$ is trivial, that is, $\tilde e=\tilde e'\subset \tilde Z\backslash\tilde Z_0$, then $\rho'=\tilde\beta^{\infty}(\tilde e)=\bar j(W_{\tilde e})$. It follows $\mathscr W=\mathscr W_e$ for $e=q(\tilde e)\in \Delta(\F)$, so claim (3) holds.

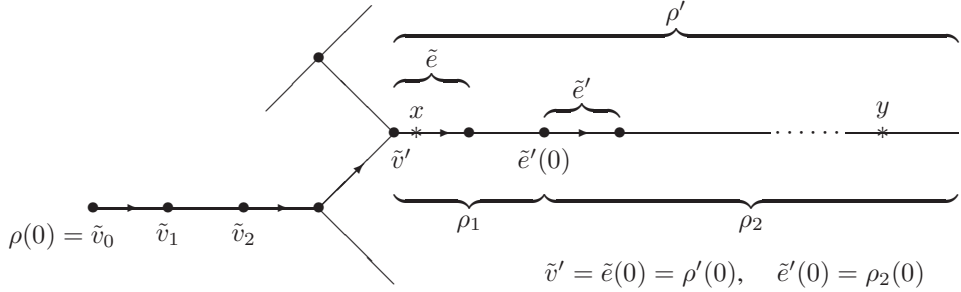
\begin{figure}[h]
\begin{center}
\setlength{\unitlength}{1mm}
\begin{picture}(115,40)(-45,-20)
\put(0,0){\line(1,0){50}}
\put(60,0){\line(1,0){15}}
\put(55.25,0){\makebox(0,0)[cc]{$\cdots\cdots$}}
\put(0,0){\line(-1,-1){10}}
\put(0,0){\line(-1, 1){10}}
\put(-17,3){\line(1, 1){14}}
\put(-40,-10){\line(1,0){30}}
\put(-10,-10){\line(1,-1){10}}
\put(0,0){\vector(1,0){7.6}}
\put(20,0){\vector(1,0){6}}
\put(-10,-10){\vector(1,1){6}}
\put(-20,-10){\vector(1,0){6}}
\put(-40,-10){\vector(1,0){6}}
\put(-40,-10){\makebox(0,0)[cc]{$\bullet$}}
\put(-30,-10){\makebox(0,0)[cc]{$\bullet$}}
\put(-20,-10){\makebox(0,0)[cc]{$\bullet$}}
\put(-10,-10){\makebox(0,0)[cc]{$\bullet$}}
\put(0,0){\makebox(0,0)[cc]{$\bullet$}}
\put(-10,10){\makebox(0,0)[cc]{$\bullet$}}
\put(10,0){\makebox(0,0)[cc]{$\bullet$}}
\put(20,0){\makebox(0,0)[cc]{$\bullet$}}
\put(30,0){\makebox(0,0)[cc]{$\bullet$}}
\put(3,0){\makebox(0,0)[cc]{$\ast$}}
\put(65,0){\makebox(0,0)[cc]{$\ast$}}
\put(-37,-12){\makebox(0,0)[rt]{$\rho(0) = \tilde v_0$}}
\put(-30,-12){\makebox(0,0)[ct]{$\tilde v_1$}}
\put(-20,-12){\makebox(0,0)[ct]{$\tilde v_2$}}
\put(1,-2){\makebox(0,0)[ct]{$\tilde v'$}}
\put(3,2){\makebox(0,0)[cb]{$x$}}
\put(65,2){\makebox(0,0)[cb]{$y$}}
\put(20,-2){\makebox(0,0)[ct]{$\tilde e'(0)$}}
\put(5,7){\makebox(0,0)[cc]{$\overbrace{\hspace{1cm}}$}}
\put(5,10){\makebox(0,0)[cc]{$\tilde e$}}
\put(25,3){\makebox(0,0)[cc]{$\overbrace{\hspace{1cm}}$}}
\put(25,6){\makebox(0,0)[cc]{$\tilde e'$}}
\put(37.5,13){\makebox(0,0)[cc]{$\overbrace{\hspace{7.5cm}}$}}
\put(37.5,16){\makebox(0,0)[cc]{$\rho'$}}
\put(10,-9){\makebox(0,0)[cc]{$\underbrace{\hspace{2cm}}$}}
\put(10,-12){\makebox(0,0)[cc]{$\rho_1$}}
\put(47.5,-9){\makebox(0,0)[cc]{$\underbrace{\hspace{5.5cm}}$}}
\put(47.5,-12){\makebox(0,0)[cc]{$\rho_2$}}
\put(20,-17){\makebox(0,0)[lt]{$\tilde v'= \tilde e(0) = \rho'(0)$, \quad $\tilde e'(0)=\rho_2(0)$}}
\end{picture}
\end{center}
\caption{The ray $\rho$ and its sub-ray $\rho'$}\label{fig: ray}
\end{figure}
\end{proof}

\subsection{RTT map of Type 1 and Type 2}\label{subsect. RTT type 1 and 2}

Let $\beta:(Z,Z_0)\to (Z,Z_0)$ be an RTT map, and $\beta_0:=\beta|_{Z_0}$.
If $\beta$ is of Type 1, then $\beta_{\pi}|_{\pi_1(Z_0, v)}=(\beta_0)_{\pi}:\pi_1(Z_0, v)\to \pi_1(Z_0, v)$ and $\beta_{\pi}(\pi_1(Z,v))\subset\pi_1(Z_0, v)$.
So $\fix(\beta_{\pi})=\fix(\beta_0)_{\pi}$ and $a(\beta_{\pi})=a(\beta_0)_{\pi}$ immediately.

For an RTT map of Type 2, we have

\begin{lem}\label{lem. RTT typer 2}
Let $\beta:(Z,Z_0)\to (Z,Z_0)$ be an RTT map of Type 2. Then every attracting fixed word of $\beta_{\pi}$ is equivalent to a one of $(\beta_0)_{\pi}$.
\end{lem}

\begin{proof}
Pick a lifting $\tilde\beta:\tilde Z\to \tilde Z$ as in Section \ref{subsect. RTT type 3}.

For an attracting fixed word $W$ of $\beta_{\pi}$, by the above Assertion $\star$, there exists an attracting fixed point $\bar j(W)\in \partial \tilde Z$ of $\tilde \beta$. Suppose $\rho: [0,+\infty]\to \tilde Z$ is a ray representing $\bar j(W)$ with origin $\rho(0)=\tilde v_0$. Since $\bar j(W)$ is attracting (under the metric $d$ on $\tilde Z$ that every edge has length $1$), as in the proof of claim (3) in Lemma \ref{key lemma: expending edges are infinite words}, we can show the set $S_{\rho}:=\rho\cap\fix\tilde \beta$ is finite, and let $\rho'$ be the sub-ray of $\rho$ such that the origin $\rho'(0)=\tilde v'$ is the only fixed point in $\rho'$.

Let $\tilde e$ be the initial oriented edge of $\rho'$ with $\tilde e(0)=\tilde v'$,  as in the proof of claim (3) in Lemma \ref{key lemma: expending edges are infinite words}, we also have $D\tilde\beta(\tilde e)=\tilde e$ and $[\tilde \beta(\rho')]=\rho'$. Recall that $\beta$ is an RTT of Type 2, if $\tilde e\subset\tilde Z\backslash\tilde Z_0$, then $Z\backslash Z_0$ has only one edge $q(\tilde e)$, and hence $\tilde\beta(\tilde e)=\tilde e$, contradicting that the origin $\tilde v'$ is the only fixed point in $\rho'$. Therefore, $\tilde e\subset \tilde Z_0$.

If $\rho'\not\subset \tilde Z_0$, it contains an oriented edge $\tilde e'\subset\tilde Z\backslash \tilde Z_0$. Since $\tilde \beta(\tilde e')\subset\tilde Z\backslash \tilde Z_0$ and $\tilde \beta(\tilde Z_0)\subset \tilde Z_0$, by the same argument as in the proof of claim (3) in Lemma \ref{key lemma: expending edges are infinite words}, we can show
$\tilde \beta$ has another fixed point but not $\tilde v'$ in $\rho'$, contradicting that $\tilde v'$ is the only fixed point in $\rho'$. Therefore, the sub-ray $\rho'\subset \tilde Z_0$,
and represents an attracting fixed point $\bar j(W')$ of $\tilde\beta_0$ for $W'$ an attracting fixed word of $(\tilde\beta_0)_{\pi}$. It follows that the attracting fixed point $W$ of $\tilde \beta$ is equivalent to $W'$. So Lemma \ref{lem. RTT typer 2} is true.
\end{proof}

\subsection{Property of $\ichr(\F)$ of RTT map}

In this subsection, we give some relations between $\ichr(\beta_0,\F_0)$ and $\ichr(\beta,\F)$.

\begin{prop}\label{key prop. on ichr}
Let $\beta:(Z,Z_0)\to (Z,Z_0)$ be an RTT map as in Thereom \ref{Thm BH}, and $\beta_0:=\beta|_{Z_0}$. If there exists an indivisible $\beta$-Nielsen
path $p$ that intersects $Z \backslash Z_0$, it is unique. (In Type 2 when $Z \backslash Z_0$ is a single edge $e$,
take $p=e$, and consider that $e$ gets initially expanded
along itself on one tip and shrinks on the other.) There are three possible cases.

$\mathrm{(i)}$ No such path $p$ exists (as always in Type 1, in Type 2 when $Z \backslash Z_0$ has more
than one edge, and possibly in Type 3). Then the $\beta$-fixed point classes are the
same as the $\beta_0$-fixed point classes, and
$$\rk(\beta, \F_0)=\rk(\beta_0,\F_0),\quad a(\beta, \F_0)=a(\beta_0,\F_0)+\delta(\F_0)$$
for all $\beta_0$-fixed point classes $\F_0$.

$\mathrm{(ii)}$ The path $p$ connects two different $\beta_0$-fixed point classes $\F'_1$
and $\F'_2$. Then the $\beta$-fixed point classes are the same as the $\beta_0$-fixed point classes, except
that $\F'_1$ and $\F'_2$ combine into a single $\beta$-fixed point class $\F'=\F'_1\cup \F'_2$. We have
$$\rk(\beta, \F')=\rk(\beta_0,\F'_1)+\rk(\beta_0,\F'_2), ~~a(\beta, \F')=a(\beta_0,\F'_1)+a(\beta_0,\F'_2)+\delta(\F')-1,$$
and for each $\beta$-fixed point class $\F_0\neq \F'$,
$$\rk(\beta, \F_0)=\rk(\beta_0,\F_0),\quad a(\beta, \F_0)=a(\beta_0,\F_0)+\delta(\F_0).$$

$\mathrm{(iii)}$ The path $p$ has both endpoints in a $\beta_0$-fixed point class $\F'_0$. Then the $\beta$-fixed point
classes are the same as the $\beta_0$-fixed point classes, and
$$\rk(\beta, \F'_0)=\rk(\beta_0,\F'_0)+1, \quad a(\beta, \F'_0)=a(\beta_0,\F'_0)+\delta(\F'_0)-1,$$
and for each $\beta$-fixed point class $\F_0\neq \F'_0$,
$$\rk(\beta, \F_0)=\rk(\beta_0,\F_0),\quad a(\beta, \F_0)=a(\beta_0,\F_0)+\delta(\F_0).$$
\end{prop}

\begin{proof}
The proof of the equalities of $a(\cdot)$ relies mainly on Lemma \ref{key lemma: expending edges are infinite words} and Lemma \ref{lem. RTT typer 2}. Although the equalities of $\rk(\cdot)$ have been proved in \cite[Corollary BH]{JWZ}, to be referenced easily, we state them in the following.

At first, let's review some notations.

For a fixed point $v$ of $\beta$, as in Definition \ref{defn of fpc by route}, let $\beta_v: \pi_1(Z,v)\to \pi_1(Z,v)$ be the induced natural injective endomorphism and $(\beta_0)_v=\beta_v|_{\pi_1(Z_0,v)}: \pi_1(Z_0,v)\to\pi_1(Z_0,v)$. For $\phi\in\{\beta_v, (\beta_0)_v\}$, let $\mathcal A(\phi)$ be the set of attracting fixed words of $\phi$. The set $\mathscr A(\phi)=\mathcal A(\phi)/\sim$ is the equivalence classes of attracting fixed words of $\phi$ (where $W, V\in \mathcal A(\phi)$ are equivalent if and only if $W=\omega V$ for some $\omega\in \fix(\phi)$). Clearly, the inclusion $\iota: (Z_0, v)\hookrightarrow (Z,v)$ is an inj-morphism from $\beta_0: (Z_0, v)\to (Z_0, v)$ to $\beta: (Z, v)\to (Z, v)$, then by Fact 2 in Section \ref{subsect. invariance}, $\fix(\beta_0)_v\subset \fix(\beta_v)$ and $\mathcal A(\beta_0)_v\subset\mathcal A(\beta_v)$.

Moreover, for two attracting fixed words $W, V\in \mathcal A(\beta_0)_v$ that are in the same equivalence class in $\mathscr A(\beta_v)$, i.e., $W=\omega V$ for some $\omega\in \fix(\beta_v)$, since $W, V\in \partial\pi_1(Z_0,v)$ and $\pi_1(Z,v)=\pi_1(Z_0,v)\ast G$ for $G$ a free subgroup of $\pi_1(Z,v)$, we have $\omega\in \fix(\beta_v)\cap\pi_1(Z_0, v)=\fix(\beta_0)_v$. So $W, V$ are also in the same equivalence class in $\mathscr A(\beta_0)_v$. Therefore, the inclusion $\mathcal A(\beta_0)_v\subset\mathcal A(\beta_v)$ induces an injection
\begin{equation}\label{inclusion induces injection}
\bar\iota_v: \mathscr A(\beta_0)_v\hookrightarrow \mathscr A(\beta_v).
\end{equation}

Recall that $a(\beta,\F)=a(\beta_v)=\#\mathscr A(\beta_v)$ for $v\in\F$.  \\

Case (i) Since every Nielsen path is a product
of indivisible Nielsen paths, it lies in $Z_0$. So for any $v\in \F_0$, $\fix(\beta_v)=\fix(\beta_0)_v$ and $\rk(\beta,\F_0)=\rk(\beta_0,\F_0)$.

If $\beta$ is of Type 1 or Type 2, then $\delta(\F_0)=0$ and $\mathscr A(\beta_v)=\mathscr A(\beta_0)_v$ by Lemma \ref{lem. RTT typer 2};
if $\beta$ is of Type 3, then by Lemma \ref{key lemma: expending edges are infinite words}, the map $e\mapsto \mathscr W_e$ defines a bijection from $\Delta(\F_0)$ to $\mathscr A(\beta_v)\backslash\mathscr A(\beta_0)_v$. In either case, we would have $a(\beta, \F_0)=a(\beta_0,\F_0)+\delta(\F_0)$.\\

Case (ii) For the  equalities of $\F_0$, the situation is the same as the one in Case (i).
Now we consider $\F'=\F'_1\cup \F'_2$. Suppose the path $p$ goes from $a \in\F'_1$ to $b \in \F'_2$.

For the $\rk(\beta,\F')$ equation it suffices to show that the natural homomorphism
$$\eta: \fix(\beta_0)_a\ast \fix(\beta_0)_b\to \fix(\beta_a), \quad [u]\mapsto [u], ~~[v]\mapsto p_\#[v]:=[pv\bar p],$$
is an isomorphism, where $u$ and $v$ are Nielsen paths in $Z_0$ at $a$ and
$b$, respectively, and $[\cdot]$ denotes loop class.
First observe that $\eta$ is injective. In fact, the $\eta$-image of any nontrivial element of the
left hand side is represented by a product $w=u_1pv_1\bar p\cdots u_lpv_l\bar pu_{l+1}$, where $u_i$
and $v_i$ are Nielsen paths in $Z_0$ at $a$ and $b$, respectively. (We allow that $u_1$ and $u_{l+1}$
be trivial, but assume the others are nontrivial.) By Property of Type 2 or 3(c) in Theorem \ref{Thm BH},
both tips of p are in $Z\backslash Z_0$, so $w$ is an immersed Nielsen path and represents a
nontrivial element in the left hand side.

On the other hand, any nontrivial element of the right hand side is represented by a
concatenation of $\beta$-Nielsen paths, hence by a product like the $w$ above, so $\eta$ is
surjective. Thus the desired isomorphism is established.\\

Below we prove the equalities of $a(\beta,\F')$.

By the bijective correspondence Assertion $\star$, two attracting fixed words $W\in \mathcal A(\beta_0)_a$ and $W'\in \mathcal A(\beta_0)_b$ can be represented by two rays $\rho, \rho'\subset\tilde Z_0:=q^{-1}(Z_0)\subset\tilde Z$ with origin $\tilde a, \tilde b\in \fix\tilde \beta$, respectively. So the inclusions induce two injections
$$\bar \iota_a: \mathscr A(\beta_0)_a\hookrightarrow \mathscr A(\beta_a), ~~~\mathscr W_\rho\mapsto \mathscr W_\rho,$$
$$\bar \iota_b: \mathscr A(\beta_0)_b\hookrightarrow \mathscr A(\beta_a),~~~\mathscr W_{\rho'}\mapsto \mathscr W_{\tilde p\rho'},$$
where $\mathscr W_\rho$ denotes the equivalence class of the attracting fixed word represented by $\rho$, and $\bar \iota_b$ can be thought as a composition $\mathscr A(\beta_0)_b\hookrightarrow \mathscr A(\beta_b)\longleftrightarrow \mathscr A(\beta_a)$, which is well-defined because the lifiting $\tilde p$ of $p$ is a $\tilde \beta$-Nielsen path from $\tilde a$ to $\tilde b$ and hence $\beta(\tilde p\rho')=\tilde p\rho'$ represents an attracting fixed word in $\mathcal A(\beta)_a$.

Furthermore, we claim $\bar\iota_a(\mathscr A(\beta_0)_a) \cap\bar\iota_b(\mathscr A(\beta_0)_b)=\emptyset$. It implies an injection
\begin{equation}\label{equa. injction}
\bar\iota:=\bar\iota_a\sqcup \bar\iota_b: \mathscr A(\beta_0)_a\sqcup \mathscr A(\beta_0)_b\hookrightarrow \mathscr A(\beta_a).
\end{equation}

Otherwise, there exist rays $\rho, \rho'\subset\tilde Z_0$ with $\tilde\beta$-fixed origins $\tilde a\in q^{-1}(a)$ and $\tilde b\in q^{-1}(b)$ respectively, such that $\rho$ and $\tilde p\rho'$ represent two equivalent attracting fixed words in $\mathcal A(\beta_a)$, that is, there exist an word $w\in \fix(\beta_a)$, such that the two rays $\rho$ and $w(\tilde p\rho')$ represent the same end in $\partial\tilde Z$.
It implies that $\gamma:=\rho\cap w(\tilde p\rho')\subset\tilde Z_0$ is a ray. Note that $w(\tilde p\rho')=w(\tilde p)\cdot w(\rho')$ and $\tilde p$ is a lifting of the indivisible Nielsen path $p$, then the two tips of $w(\tilde p)\subset \tilde Z\backslash\tilde Z_0$ and $\gamma\subset w(\rho')\subset\tilde Z_0$. Let $\rho_1:=\rho\backslash\gamma$ and $\rho'_1:=w( \rho')\backslash\gamma$, then $\rho_1\bar\rho'_1\subset \tilde Z_0$ is a path from $\tilde a$ to $w(\tilde b)$. Recall $w\in \fix(\beta_a)$ and $\tilde b\in\fix\tilde\beta$, then $\tilde\beta(w(\tilde b))=\beta_a(w)(\tilde\beta(\tilde b))=w(\tilde b)$ is a fixed point of $\tilde\beta$. It implies that $\rho_1\bar\rho'_1\subset \tilde Z_0$ is a Nielsen path from $\tilde a$ to $w(\tilde b)$, contradicting that $a\in\F'_1$ and $b\in \F'_2$ belong to distinct fixed point classes of $\beta_0$. So the claim is true.\\

Using the injection $\bar\iota$ in equation (\ref{equa. injction}), we can compute $a(\beta,\F')$ immediately. There are two cases:

(a) $p$ is the unique edge $e$ of $Z\backslash Z_0$ in Type 2 (recall that in this case $\delta(\F')=1$ because $e$ gets initially expanded
along itself on only one tip). By Lemma \ref{lem. RTT typer 2}, $\bar\iota$ is a bijection. So
$$a(\beta, \F')=a(\beta_0,\F'_1)+a(\beta_0,\F'_2)=a(\beta_0,\F'_1)+a(\beta_0,\F'_2)+\delta(\F')-1.$$

(b) $p$ is the unique indivisible Nielsen path in Type 3. Follows from Lemma \ref{key lemma: expending edges are infinite words}, each element of $\mathscr A(\beta_a)\backslash\bar\iota(\mathscr A(\beta_0)_a\sqcup \mathscr A(\beta_0)_b)$ can be represented by $\mathscr W_e$ for a unique edge $e\in \Delta(\F')$, except that $e_1, e_2\in \Delta(\F')$ are the two tips of the indivisible Nielsen path $p$ and they give the same element $\mathscr W_{e_1}=\mathscr W_{e_2}$, and conversely. Therefore, we also have
$$a(\beta, \F')=a(\beta_0,\F'_1)+a(\beta_0,\F'_2)+\delta(\F')-1.$$\\

Case (iii) For the  equalities of $\F\neq\F'_0$, the situation is the same as the one in Case (i).

Suppose the path $p$ goes from $a$ to $b$ both in $\F'_0$.
By the similar argument as in case (ii), we can show $\eta$
$$\eta: \fix(\beta_0)_a\ast J\to \fix(\beta_a), ~~\mathrm{defined ~by} ~[u]\mapsto [u], ~~[t]\mapsto [t] $$
is an isomorphism, where the loop class $[u]\in\fix(\beta_0)_a$ and $J$ is the infinite cyclic group generated by the loop class $[t]$ represented by the loop
$t:=pq$ when $a\neq b$, but taking $t:=p$ when $a=b$. It implies $\rk(\beta,\F'_0)=\rk(\beta_0,\F'_0)+1$.

By equation ($\ref{inclusion induces injection}$), $\bar\iota_a: \mathscr A(\beta_0)_a\hookrightarrow \mathscr A(\beta_a)$ is an injection.
Then by the similar argument as the one in the above Case (a) and Case (b), we have $a(\beta, \F'_0)=a(\beta_0,\F'_0)+\delta(\F'_0)-1.$
\end{proof}

As a corollary of the above Proposition \ref{key prop. on ichr}, we have

\begin{cor}\label{Cor. ichr for RTT}
Let $\beta: (Z,Z_0)\to (Z,Z_0)$ be an RTT map as in Theorem \ref{Thm BH} and $\beta_0:=\beta|_{Z_0}: Z_0\to Z_0$. Then for every nonempty fixed point class $\F$ of $\beta$, we have
$$\ichr(\beta, \F)=\ichr(\beta_0, \F)-\delta(\F),$$
where $\ichr(\beta_0, \F):=\sum_{i=1}^k\ichr(\beta_0, \F_i)$ if $\F=\sqcup \F_i$ is a union of $k\leq 2$ $\beta_0$-fixed point classes $\F_i, i=1,\ldots,k$.
\end{cor}

\begin{proof}
Since $\ichr(\F)=1-\rk(\F)-a(\F)$, the conclusion can be checked immediately case by case in Proposition \ref{key prop. on ichr}.
\end{proof}

%============================================================================================================================================================================================
\section{Proofs of Theorems}\label{sect. proof of thms}

Now we prove main theorems of this paper.

\begin{proof}[\textbf{Proof of Theorem \ref{main thm 1}}]
Let $f$ be a $\pi_1$-injective selfmap of a connected finite graph $X$, and $\F$ a fixed point class of $f$. Since $\ind(\F)$ and $\ichr(\F)$ both have homotopy and commutation invariance (see Section \ref{subsect. invariance}), it suffices to prove Theorem \ref{main thm 1} for an RTT map $\beta: (Z,Z_0)\to (Z,Z_0)$ in Theorem \ref{Thm BH}, where $Z$ is a connected finite graph without vertices of valence $1$ and with $\chi(Z)=\chi(X)$. Assume $\chi(Z)\leq 0$ because it is trivial for $\chi(Z)>0$.

If $\chi(Z)=0$, i.e., $Z$ is a circle, then by Example~\ref{exam: circle}, we have $\ind(\F)=\ichr(\F)$ for every fixed point class $\F$ of $\beta$. So Theorem \ref{main thm 1} holds.

If $\chi(Z)<0$ and $\F$ is empty, then $\ind(\F)=0$ while $0\leq \ichr(\F)\leq 1$ by Proposition \ref{empty fpc has ichr 1}, and hence Theorem \ref{main thm 1} holds. Now we assume $\F$ is a nonempty fixed point class of $\beta$ in the following. Let $Z_1,\ldots, Z_n$ be the connected components of $Z_0$. Suppose the $\beta$-invariant
ones are $Z_1,\ldots, Z_k$. Denote $\beta_0:=\beta|_{Z_0}: Z_0\to Z_0$ and $\beta_i:=\beta|_{Z_i}: Z_i\to Z_i$ for $i=1,\ldots, k.$
Since $Z$ is a connected graph without vertices of valence $1$, and $Z_0$ is a proper
subgraph containing all the vertices of $Z$, it is easy to see that all $\chi(Z_i)> \chi(Z)$. So, working inductively, we may
assume that Theorem \ref{main thm 1} is true for every $\beta_i$, that is,
\begin{equation}\label{equa. inductive hypothesis}
\ind(\beta_i,\F_i)\leq\ichr(\beta_i,\F_i)
\end{equation}
for every fixed point class $\F_i$ of $\beta_i$.

For the nonempty fixed point class $\F$ of $\beta$, note that $\F$ is a union of finitely many (in fact at most 2 by Proposition \ref{key prop. on ichr}) $\beta_0$-fixed point classes $\F_i$, each is contained in a unique component $Z_i$ and is also a fixed point class of $\beta_i$, so $\ind(\beta_0,\F_i)=\ind(\beta_i,\F_i)$ and $\ichr(\beta_i,\F_i)=\ichr(\beta_0,\F_i)$. Therefore, by the inductive hypothesis (equation \ref{equa. inductive hypothesis}), we have
$$\ind(\beta_0,\F_i)\leq\ichr(\beta_0,\F_i).$$
Combine the above inequality, equation (\ref{equa. additivity of index}) with Corollary \ref{Cor. ichr for RTT}, Theorem \ref{main thm 1} holds.
\end{proof}

\begin{proof}[\textbf{Proof of Theorem \ref{main thm for figure 8}}]
Since $\ind(\F)$ and $\ichr(\F)$ both have homotopy invariance, it suffices to assume that $f$ is an RTT map $\beta: (Z,Z_0)\to (Z,Z_0)$ in Theorem \ref{Thm BH}, where $\chi(Z)=\chi(X)=-1$.

If $\F$ is inessential, then the conclusion is clear by Theorem \ref{main thm 1}.

Now we assume that $\F$ is essential, then it is a nonempty fixed point class of $\beta$. Using the same notations $Z_1,\ldots, Z_n$ and $\beta_0,\beta_1,\ldots, \beta_k$ as in the above proof of Theorem \ref{main thm 1}, we have $1\geq \chi(Z_i)> \chi(Z)=-1$, and $\F$ is a finite union of $\beta_0$-fixed point classes $\F_i$ which is contained in a unique component $Z_i$ and is also a fixed point class of $\beta_i$.

If $\chi(Z_i)=1$, then $Z_i$ is a vertex of $Z$, and hence every fixed point class $\F_i$ of $\beta_i:\Z_i\to Z_i$ consists of a single point with $\ind(\beta_i,\F_i)=\ichr(\beta_i,\F_i)=1$.
If $\chi(Z_i)=0$, after a homotopy, $Z_i$ is a circle. From Example~\ref{exam: circle}, we have
$\ind(\beta_i,\F_i)=\ichr(\beta_i,\F_i)$
for every fixed point class $\F_i$ of $\beta_i$. Therefore, we always have
$$\ind(\beta_0,\F_i)=\ind(\beta_i,\F_i)=\ichr(\beta_i,\F_i)=\ichr(\beta_0,\F_i).$$
Combine the above equality, equation (\ref{equa. additivity of index}) with Corollary \ref{Cor. ichr for RTT}, the conclusion $\ind(\F)=\ichr(\F)$ holds.
\end{proof}

\begin{rem}
From the above two proofs, it is clear that Conjecture \ref{conj} is true if and only if $\ind(\F)=\ichr(\F)=0$ holds for every empty fixed point class $\F$.
\end{rem}

\begin{proof}[\textbf{Proof of Theorem \ref{stronger version of main thm 2}}]
For an injective endomorphism $\phi$ of $F_n$, there exists a $\pi_1$-injective selfmap $f: (X, \ast)\to (X, \ast)$ of a connected finite graph $X$ with fundamental group $\pi_1(X,\ast)=F_n$ such that $\phi=f_{\ast}: \pi_1(X,\ast)\to \pi_1(X,\ast)$, $[c]\mapsto[f\comp c]$. Moreover,
for any $i_{[w]}\comp\phi\in \inn\phi$, we have
$$i_{[w]}\comp\phi=f_{w}: \pi_1(X,\ast)\to \pi_1(X,\ast), \quad [c]\mapsto[w(f\comp c)\bar w],$$
where $w$ is a loop based at $\ast$ and the loop class $[w]\in \pi_1(X,\ast)$.

Two endomorphisms $\psi'=i_{[w']}\comp\phi$ and $\psi=i_{[w]}\comp\phi$ in $\inn\phi$
are \emph{similar} (see Section \ref{subsect. similarity invariance} for a definition) if and only if there exists $[u]\in\pi_1(X, \ast)$ for $u$ a loop based at $\ast$ such that $\psi'=i_{[u]}\comp\psi\comp (i_{[u]})^{-1}$, that is,
\begin{equation}\label{equa. routes}
i_{[w']}\comp\phi=i_{[u]}\comp i_{[w]}\comp\phi\comp i_{[\bar u]}=i_{[uw(f\comp\bar u)]}\comp\phi.
\end{equation}
Since $\phi(F_n)\cong F_n$ and the centralizer of any nontrivial element in a free group $F_n$ of rank $n\geq 2$ is free cyclic, the above Equation (\ref{equa. routes}) holds if and only if
$[w']=[uw(f\comp \bar u)]$, that is, the $f$-routes $w'$ and $w$ are associated with the same fixed point class $\F$ (see Definition \ref{def. f-route class}). Therefore, we have a bijection $\F\mapsto [\psi]$ between the fixed point classes of $f$ and the similarity classes contained in $\inn\phi$, with
\begin{equation}\label{equa. ichr psi}
\ichr(\F)=1-\rk(\F)-a(\F)=1-\rk\fix(\psi)-a(\psi).
\end{equation}
By Corollary \ref{main cor}, we have
$$\sum_{[\psi]\subset\inn\phi}\max\{0,~~\rk\fix(\psi)+a(\psi)/2-1\}\leq n-1.$$
So Theorem \ref{stronger version of main thm 2} holds.
\end{proof}

\begin{proof}[\textbf{Proof of Theorem~\ref{main thm: existence of fixed words}}]
The conclusion follows from Example \ref{exam:rankone} immediately if the rank $n=1$. Now we assume $n\geq 2$. As in the above proof of Theorem \ref{stronger version of main thm 2}, there exists a $\pi_1$-injective selfmap $f: X\to X$ of a connected finite graph $X$ with fundamental group $\pi_1(X)=F_n$ such that $\phi=f_{\ast}: \pi_1(X,\ast)\to \pi_1(X,\ast)$, and the abelianization of $\phi$,
$$\phi^{\mathrm{ab}}=f_\sharp: H_1(X)\to H_1(X)=\Z^n,$$
the endomorphism induced by $f$ on the homology group $H_1(X)$. By the famous Lefschetz-Hopf fixed point theorem, the Lefschetz number
$$L(f)=1-\tr(\phi^{\mathrm{ab}})=\sum_{\F\in\fpc(f)}\ind(f,\F).$$

If $\tr(\phi^{\mathrm{ab}})<1$, then there exists $\F\in \fpc(f)$ such that $\ind(f,\F)>0$, and hence
$\ichr(f,\F)>0$ by Theorem \ref{main thm 1}. So by Equation (\ref{equa. ichr psi}), there exists $c\in F_n$ such that
$\rk\fix(i_c\comp \phi)=a(i_c\comp \phi)=0$.

If $n=2$ and $\tr(\phi^{\mathrm{ab}})>1$, then there exists $\F\in \fpc(f)$ with $\ind(f,\F)<0$, and hence
$\ichr(f,\F)<0$ by Theorem \ref{main thm for figure 8}. So by Equation (\ref{equa. ichr psi}), there exists $c\in F_n$ such that $\rk\fix(i_c\comp \phi)+a(i_c\comp \phi)>1$.
\end{proof}

\begin{proof}[\textbf{Proof of Theorem~\ref{main thm 2}}]
The conclusion follows from Example \ref{exam:rankone} immediately if the rank $n=1$. For $n\geq 2$, pick one of the summands in the inequality of Theorem~\ref{stronger version of main thm 2}, we have that $\max\{0,\rk\fix(\phi)+a(\phi)/2-1\}\leq n-1$. This is what we need to show.
\end{proof}

%============================================================================================================================================================================================
\section{Examples}\label{sect. examples}
The following example shows that the bounds in  Theorem \ref{main thm 2} and Theorem \ref{stronger version of main thm 2} are sharp.

\begin{exam}
Let $f: (R_n, \ast) \to(R_n, \ast)$ be a $\pi_1$-injective selfmap of the graph $R_n$ with one vertex $\ast$ and $n\geq 1$ edges $a_1,\ldots, a_n$, such that $f(a_i)=a^2_i$ for
$i=1,\ldots,n$.

Note that $f$ is an RTT map with $Z=R_n$ and $Z_0=\ast$. Then $\ast$ is the unique nonempty fixed point class of $f$, and $\Delta(\ast)=\{a_1,\ldots, a_n,\bar a_1,\ldots,\bar a_n\}$ of cardinality $\delta(\ast)=2n$. If we write the path class of the edge $a_i$ still by $a_i$, then the fundamental group $\pi_1(R_n, \ast)=\langle a_1,\ldots, a_n\rangle\cong F_n$, and $f_{\ast}: \pi_1(R_n, \ast)\to \pi_1(R_n, \ast)$ defined by $f_{\ast}(a_i)=a^2_i$. Note that $f_\ast$ is not an automorphism, with $\fix(f_{\ast})=\{1\}$ and $2n$ equivalence classes of attracting fixed points that are induced by the oriented edge $a_i, \bar a_i\in \Delta(\ast)$,
$$W_{a_i}=f_{\ast}^{\infty}(a_i)=a_ia_ia_i\cdots,\quad W_{\bar a_i}=f_{\ast}^{\infty}(a_i^{-1})=a^{-1}_ia_i^{-1}a_i^{-1}\cdots.$$
Therefore, $a(\ast)=a(f_{\ast})=2n$,
$\rk\fix(f_{\ast})+a(f_{\ast})/2=n$ and
$$\ind(\ast)=\ichr(\ast)=1-\rk\fix(f_{\ast})-a(f_{\ast})=1-2n.$$
\end{exam}

Next example shows that the inequality in Theorem \ref{main thm 2} can be strict.

\begin{exam}
Let $f: (R_2, \ast) \to(R_2, \ast)$ be a $\pi_1$-injective selfmap of the graph $R_2$ with one vertex $\ast$ and two edges $a_1, a_2$, such that $f(a_1)=a_1$ and
$f(a_2)=\bar a_2a_1a_2$.

Note that $f$ is an RTT map with $Z=R_2$ and $Z_0=a_1$. Then $\ast$ is a nonempty fixed point class of $f$, and $\Delta(\ast)=\{\bar a_2\}$. If we write the path class of the edge $a_i$ still by $a_i$, then the fundamental group $\pi_1(R_2, \ast)=\langle a_1, a_2\rangle\cong F_2$, with $f_{\ast}: \pi_1(R_2, \ast)\to \pi_1(R_2, \ast)$ defined by $f_{\ast}(a_1)=a_1$ and
$f_{\ast}(a_2)=a^{-1}_2a_1a_2$. Note that $f_\ast$ is not an automorphism, with $\fix(f_{\ast})=\langle a_1\rangle\cong \Z$ and the only equivalence class of attracting fixed points that is induced by the oriented edge $\bar a_2\in \Delta(\ast)$
$$W_{\bar a_2}=f_{\ast}^{\infty}(a_2^{-1})=a^{-1}_2a_1^{-1}a_2a_1^{-1}a^{-1}_2a_1a_2\cdots,$$
Therefore,
$\rk\fix(f_{\ast})+a(f_{\ast})/2=1+1/2<2$ and
$$\ind(\ast)=\ichr(\ast)=1-\rk\fix(f_{\ast})-a(f_{\ast})=-1.$$
\end{exam}

Finally, we give two examples that support Conjecture \ref{conj} for empty and nonempty fixed point classes with indices $0$.

\begin{exam}
Let $f: (R_2, \ast) \to(R_2, \ast)$ be a $\pi_1$-injective selfmap of the graph $R_2$ with one vertex $\ast$ and two edges $a, b$, such that $f(a)=b$ and $f(b)=\bar a$.

Below we will show that the fixed point class $\F_a$ associated to the $f$-route $a$ is empty with $\ind(\F_a)=\ichr(\F_a)$.

Indeed, fix a universal covering $q: \tilde R_2\to R_2$ with a given point $\tilde\ast\in q^{-1}(\ast)$, and a lifting $\tilde a: (I, 0, 1)\to (\tilde R_2, \tilde\ast,\tilde a(1))$ of the loop $a$. Then the lifting $\tilde f: \tilde R_2\to \tilde R_2$ defined by $\tilde f(\tilde \ast)=\tilde a(1)$ is a composition of a rotation and a translation, so it is an isometry of the tree $\tilde R_2$ equipped with the natural metric $d$ such that  each edge has length $1$. Note that $d(\tilde f(\tilde\ast), \tilde\ast)=d(\tilde a(1), \tilde\ast)=1$, so for any point $x\in \tilde R_2$,
$$d(\tilde f(x), \tilde \ast)=d(\tilde f(x), \tilde f(\tilde\ast))\pm d(\tilde f(\tilde\ast), \tilde\ast)=d(x, \tilde\ast)\pm 1.$$
It implies that $\tilde f(x)\neq x$ and hence $\fix \tilde f=\emptyset$, namely, the fixed point class $\F_a$ is empty.

On one hand, the $f$-route $a$ induces an injective endomorphism
$$f_a: \pi_1(R_2,\ast)\to \pi_1(R_2,\ast), \quad a\mapsto aba^{-1}, ~~b\mapsto a^{-1},$$
with $\fix(f_a)=\langle aba^{-1}b^{-1}\rangle\cong \Z$. On the other hand, consider $f_{\ast}:\pi_1(R_2,\ast)\to \pi_1(R_2,\ast)$ defined by $a\mapsto b, b\mapsto a^{-1}$. Then for any infinite word $W\in \partial\pi_1(R_2,\ast)$ with initial segment $|W_i|=i$, the word length $|f_{\ast}(W_i)|=i$ , and $$|f_a(W_i)|=|af_{\ast}(W_i)a^{-1}|\leq i+2.$$
It follows that $f_a$ has no attracting fixed words. Therefore, $a(f_a)=0$.

By the arguments above, for the empty fixed point class $\F_a$ defined by the $f$-route $a$, we have
$$\ind(\F_a)=\ichr(\F_a)=1-\rk\fix(f_a)-a(f_a)=0.$$
\end{exam}

\begin{exam}[\cite{Jiang1984}]
Let $f: (R_2, \ast) \to(R_2, \ast)$ be a $\pi_1$-injective selfmap of the graph $R_2$ with one vertex $\ast$ and two edges $a, b$, such that the induced endomorphism of $\pi_1(R_2, \ast)$ is given by  $f_*(a)=a^{-1}$ and $f_*(b)=a^{-1}b^2$.

It is known from \cite{Jiang1984} that $f$ has two nonempty fixed point classes and both of them have indices zero and hence both are inessential.

Note that $\fix(f_{*})=\{1\}$ and $f_{*}$ has an attracting fixed word
$$f_{*}^{\infty}(b^{-1})=b^{-2}ab^{-4}ab^{-2}a\cdots.$$
Hence, $a(f_\ast)=1$ and the nonempty fixed point class consisting of $\{\ast\}$ has
$$\ind(\ast)=\ichr(\ast)=1-\rk\fix(f_*)-a(f_*)=0.$$
\end{exam}

%%%==================================================================================================================================================================================

\end{document}